\documentclass[leqno,12pt]{article} 
\usepackage{amsmath, amssymb}
\usepackage{amsthm} 
\def\N{{{\mathbb N}}}

\def\Z{{{\mathbb Z}}}
\def\Q{{{\mathbb Q}}}
\def\R{{{\mathbb R}}}
\def\C{{{\mathbb C}}}
\def\P{{{\mathbb P}}}

\def\O{{{\mathcal O}}}
\def\1C{{{\mathcal C}}}
\def\E{{{\mathcal E}}}

\def\k{{{\kappa}}}
\def\d{{\rm d}}
\def\v{{\bf v}}

\def \e{{\varepsilon}}
\def\c{{{ c}}}

\DeclareMathOperator{\ord}{ord}

\usepackage{verbatim}

\theoremstyle{plain} 
\newtheorem{theorem}{\indent\sc Theorem}[section]
\newtheorem{lemma}[theorem]{\indent\sc Lemma}
\newtheorem{corollary}[theorem]{\indent\sc Corollary}
\newtheorem{proposition}[theorem]{\indent\sc Proposition}

\theoremstyle{definition} 
\newtheorem{definition}[theorem]{\indent\sc Definition}
\newtheorem{remark}[theorem]{\indent\sc Remark}
\newtheorem{example}[theorem]{\indent\sc Example}


\author{Pietro Corvaja, Julian Demeio, David Masser and Umberto Zannier}
\title{ On the torsion values for sections of an elliptic scheme}

\begin{document}

\date{}\maketitle

\noindent{\sc Abstract.}   We shall consider sections of an elliptic scheme $\E$  over an affine  base  curve $B$, and study the points of  $B$  where the section takes a torsion value. In particular, we shall relate the distribution in $B$ of these points with the canonical height of the section, proving an integral formula involving a measure on $B$ coming from the so-called Betti map of the section. We shall show that this measure is the same one which appears in dynamical issues related to  the section.

This analysis will  also involve the {\it multiplicity}  with which a torsion value is attained, which is an independent problem. We shall prove finiteness theorems for the points where the multiplicity is higher than expected. Such multiplicity has also a relation with Diophantine Approximation and quasi-integral points on $\E$ (over the affine ring of $B$), and in sections \ref{Sec4} and \ref{cmb} of the paper we shall exploit this viewpoint, proving an effective result in the spirit of Siegel's theorem on integral points. 













\begin{section}{Introduction}

    In this paper we deal with an elliptic scheme $\mathcal{E}\to B$ over an affine algebraic  curve $B$,     and defined for instance  by a Weierstrass equation
    \begin{equation}\label{E.W}
  \mathcal{E}:\, y^2 = (x-\alpha_1)(x-\alpha_2)(x-\alpha_3),
\end{equation}
where $\alpha_i$, $i=1,2,3$ 
 lie in $ \C(B)$, the origin being the point at infinity.  

For $b\in B$, we let $\mathcal{E}_b$ be the fiber over $b$. We can suppose that  
the $\alpha_i$ are regular in $\C(B)$ such that
the  discriminant of the cubic  polynomial  on the right of  \eqref{E.W}  never vanishes on $B$ (i.e. $(\alpha_i-\alpha_j)$ do not vanish for $i\neq j$), so that $\mathcal{E}_b$ will really be an elliptic curve for each $b$.

We shall disregard the case when this scheme is isotrivial, which is equivalent to the $j$-invariant of the elliptic curve being constant. Morevoer, at some points we shall also suppose that  the elliptic scheme is defined over the field of algebraic numbers. \smallskip

A  nonzero section $\sigma:B\to\E$ of this scheme can be written as $\sigma=(x_\sigma,y_\sigma)$ where $x_\sigma,y_\sigma \in\C(B)$ satisfy the above equation. 

Our main interest lies in what we call the   {\it torsion points}  of $\sigma$, namely the set of points of the base $B$ where $\sigma$ takes a value which is torsion in the appropriate fiber: 
\begin{equation}\label{E.T}
  T_\sigma = \{ b\in B 
  \, |\, \exists n>0, n\cdot  \sigma(b) =0 \}.
\end{equation}

Of course one usually refers to torsion points as points of $\E$, however our present terminology should not risk to create any confusion.


\subsection{Generalities on $T_\sigma$}

\underline{\it $T_\sigma$ is large}:  
We pause to remark that $T_\sigma$ is always infinite. When $\sigma$ is torsion as a section, of course $T_\sigma=B$, and  this is a trivial case usually disregarded here, so assume that $\sigma$ is non-torsion.
Then the issue  is maybe less obvious then it may seem. Here are some arguments.

Writing $n\cdot \sigma=(x_{n\sigma}, y_{n\sigma})$, a torsion point (in our sense) corresponds to a pole in $B$ of some $x_{n\sigma}$. Saying that $T_\sigma$ is infinite amounts to the fact that the set of these poles is  infinite as $n$ varies.  Siegel's theorem over function fields implies this fact and more. Or we can use the $abc$ inequality over function fields, after observing that   
 for each $n>0$  the functions $x_{n\sigma}$  have all poles in $T_\sigma$ and the functions $x_{n\sigma}-\alpha_i$  have all zeros in $T_\sigma$.    (See for instance  \cite{CZ} and    \cite{Z} for a bit more detail.  This link with Siegel's theorem will be central in Section \ref{Sec4}).

There is also an analytic argument employing the so called Betti map, which will be central in this paper; this gives even the stronger result that the torsion points are dense for the complex topology of $B$.\footnote{As proved in \cite{LZ}, this is not true for the $p$-adic topology.}  (See again \cite{CZ} for more on this, and  \cite{Z} for a further argument using reduction modulo $p$.)

A recent paper of De Marco-Mavraki \cite{DMM}, in the case when $\overline{\mathbb{Q}}$ is a field of definition, proves even a property of  Galois equidistribution of such points. Roughly speaking this asserts that there is a probability measure on the base $B$ such that, as the torsion order grows to infinity, the Galois conjugates of a torsion point `tend' to be equidistributed in $B$ with respect to this measure, i.e. every prescribed sufficiently regular region contains a percentage of conjugates approximately equal to its measure. 
 They work even with $p$-adic  valuations, but if we stick to the archimedean case, we shall see later how the measure in this case comes  naturally from the Betti map, so their results very well fit into our study. 

\medskip

\underline{\it $T_\sigma$ is small}: In the opposite direction, one may prove that these points are {\it sparse} in certain meanings. For instance, when the scheme and the sections are defined over $\overline\Q$, a theorem of Silverman-Tate predicts bounded height for such points (provided the section is not torsion), hence their degree  (over $\Q$) tends to infinity.  (See \cite{Z}, especially Appendix C by the third author, for a self-contained proof.) We may add to this $p$-adic integrality constraints of the type appearing in the theorem of Lutz-Nagell (see  \cite{silverman}, VIII, \S 7.)

And if we seek points which are torsion simultaneously for another section which is linearly independent with the former, then one may prove their  finiteness over the whole algebraic closure, i.e. disregarding their degree. See e.g. the paper \cite{CMZ}, which extends to $\C$ such results, previously proved over $\overline\Q$ in papers quoted therein, and see also \cite{Z}, Ch. 3. (These finiteness theorems are cases of the so-called {\it Pink-Zilber conjecture}.)

\subsection{Our issues}  

\subsubsection{Multiplicities} We have just remarked that  the poles of the rational functions $x_{n\sigma}$ for varying $n$ are the elements of $T_\sigma$. So, if for instance we are interested in {\it counting}  the elements of  $T_\sigma$ of a given torsion order $m$, and if we want to do that by comparison with the degree $\deg x_{m\sigma}$, then we must take into account the {\it multiplicity} with which they appear as poles of  $x_{m\sigma} $.  

Now, the coordinate function $x$ for a Weierstrass equation of an elliptic curve has a pole of order $2$ at infinity, hence the  multiplicity  for the said poles will be even. But we should not expect  multiplicity higher than $2$, except for special points. 

\medskip 

Thus an issue here is to  seek to {\it bound from above} such multiplicity, and to {\it   discover  the distribution} of torsion points on $B$ where high multiplicity occurs. 



This  will be the first of the problems discussed in the paper. We shall start by describing the multiplicity in different equivalent terms. As stated in  Lemma \ref{L.local} below, this multiplicity corresponds to the local multiplicity intersection of the image of the section with a torsion curve (i.e. defined by an equation $np=0$, $p\in \E$, for an integer $n>0$).

We shall analyse the multiplicity by means of the Betti map of the section, recalled in \S \ref{S.Betti}. This shall lead to finiteness theorems for the points in $T_\sigma$ where the multiplicity is higher than expected, and also to a bound for the maximal multiplicity. Actually, the analysis will prove a stronger result on letting the section vary through the full Mordell-Weil group of sections. \footnote{Very recently D. Ulmer and G. Urzua kindly   informed us about their preprint \cite{UU} in which, among other things,  they proved the finiteness instance in the case of a cyclic group, using somewhat different methods.}

We also note that a high multiplicity  at a point $b\in B$, as a pole of $x_{n\sigma}$, means that $n\sigma$ is near to infinity (i.e. the origin of $\E$)  in the valuation of $\C(B)$ associated to $b$. This fact  links the issue with Diophantine Approximation over function fields. Then we shall take this viewpoint in section \ref{Sec4} (see below for more).

\subsubsection{Canonical height of a section,  torsion points and the Betti measure }

We shall also investigate the distribution of $T_\sigma$ in regions of $B$; this will be done by  counting asymptotically the points of order $n$ with varying $n$. 

After proving that the multiplicity at these points will be generally what is expected (i.e. $2$ as a pole of  $x_{n\sigma}$), the counting of points will give essentially the degree of $x_{n\sigma}$, which is asymptotically $n^2\hat h(\sigma)$, where $\hat h$ is a canonical height associated to twice  the divisor at infinity on $\E$.

But the counting of points can be also done through the Betti map. For given $n$, we have to count the number of points where the Betti map takes a rational value with denominator (dividing) $n$. For large $n$ this will be essentially $n^2$ times the area of the base, with respect to the measure obtained locally by pulling back the Lebesgue measure on $\R^2$ by the Betti map.

By comparing these two approaches, we shall also obtain a certain integral formula for the canonical height of the section.

\medskip  

It is to be noted that, beyond the control on multiplicities, we shall need control on the behaviour of the Betti map near the points  of bad reduction of $\E$ (for instance, $0,1,\infty$ for the Legendre elliptic scheme). This will be done by using that the Betti map is {\it definable}, as recently proved by G. Jones and H. Schmidt \cite{JS}.  (This definability will be used actually also in the previous part of the paper.) To show that the counting can be done through the usual comparison with an area it will be necessary to use a further result by Barroero and Widmer \cite{BW} proving the required asymptotic for definable maps.

We shall also indicate other possible methods for proofs of these results.

This section will further contain an Appendix, showing that the measure induced by the Betti map is the same as the one used by De Marco and Mavraki in the above quoted context. As a byproduct of this verification, we shall obtain a certain characterisation of  the possible measures which could appear {\it a priori}.



\subsubsection{Effective analysis of integral points on  $\E$ and multiplicities again}

As mentioned above,  the issue of multiplicities for the poles of the rational  functions $x_{n\sigma}$, for $n\in\N$, are related to diophantine approximation and quasi integral points in the function field of $B$.  In Section \ref{Sec4} of the paper we will  exploit this viewpoint.

We will first prove a version of effective Roth's theorem over function fields, by adapting a method of J. Wang \cite{W}.  Then we shall use Siegel's method for quasi-integral points on an elliptic curve, to deduce the sought approximation result.

Concerning multiplicities, this will yield results in a sense  much weaker than those coming from the Betti map if applied to varying elements in a prescribed finitely generated group of sections. However these results will be stronger if applied individually to an arbitrary section. Morevoer, they will be completely effective and uniform, and  admit applications  beyond  what follows from  the former methods.   
This analysis in accomplished by examples illustrating the conclusions in all details.

\end{section}



\section{The Betti map}\label{S.Betti}

\begin{definition}\label{bettidef}
	Let $\pi:\mathcal{E} \rightarrow B$ be an elliptic scheme (with $B/\C$ a smooth projective curve) with bad reduction locus $S \subset B$, and let $D \subset B \setminus S$ be a simply connected domain. Let $(\rho_1(\lambda),\rho_2(\lambda)), \ \lambda \in D$, be a holomorphic choice of periods for the elliptic logarithm. For any $P \in \pi^{-1}(D)$, we call the \textit{Betti coordinates} of $P$, and denote them by $\beta(P) := (\beta_1(P),\beta_2(P)) \in (\R/\Z)^2$, the unique elements of $\R/\Z$, such that the following equality holds:
	\begin{equation}\label{Eq:Betti_def}
	\log^{ab}(P)=\beta_1(P)\rho_1(\pi(P))+\beta_2(P)\rho_2(\pi(P)),
	\end{equation}
	where $\log^{ab}$ denotes the abelian logarithm multivalued function. We call the \textit{Betti map} the (real analytic) map that associates to a point $P$ its Betti coordinates.
\end{definition}

We note that, although for the Betti map $(\beta_1,\beta_2)$ to be well-defined  one needs to restrict oneself to a simply connected domain, the $1$-forms $\d \beta_1$, $\d \beta_2$, which are going to appear commonly in this paper, are well-defined on $\pi^{-1}(D)$, for any simply connected open $D \subset \P_1 \setminus S$. Moreover, the $2$-form $\d \beta_1 \wedge \d \beta_2$ is well-defined on $\mathcal{E} \setminus \pi^{-1}(S)$, i.e. it is independent of the local choice of periods.

\smallskip

If we happen to have a specific section $\sigma:B \rightarrow \mathcal{E}$, we sometimes call, with a slight abuse of notation, when there is no risk of confusion, the Betti coordinates of a point $P \in B$, the Betti coordinates of the point $\sigma(P) \in E_B$ (after a specific choice of a simply connected domain $D \subset B$ has been made). It may be continued to all of $B(\C)$, with monodromy transformation that we forget about at the moment. 
(See e.g. \cite{CMZ} for a simple description sufficient for our task here, and see \cite{ACZ} for much  more precise information about the Betti map, in any dimension.)

For a more detailed exposition about the Betti map, see e.g. \cite[Section 1.1]{CMZ}.



\begin{remark}\label{Rmk:Surjectivity}
	The Betti map of a non-torsion section is proved to be generically submersive (see \cite{Z}, \S 2.5 or \cite{ACZ}, \cite{CMZ}); however it is not necessarily surjective to $(\R/\Z)^2$. In fact, for every positive integer $N$ there exists always an (algebraic) section of the Legendre scheme such that the Betti map associated to it ``misses'' all the rational points with denominator dividing $N$. One such example is any algebraic section $\sigma_N$ such that $[N] \cdot \sigma_N=\sigma_0$, where, in the standard Weierstrass notation of the Legendre scheme (recalled in Section \ref{Sec:Legendre}), $\sigma_0(\lambda)=(2,\sqrt{2(2-\lambda)})$.
\end{remark}

\begin{definition} \label{D.multB}  We define the multiplicity of $\beta$ at a point $b\in B(\C)$ as the minimum order of a partial derivative of $\beta$ which does not vanish at $b$. This notion is clearly independent of the local determination for the Betti map.
\end{definition}

For completeness, we recall from \cite{CMZ} that the rank of the (differential of the) Betti map (at any point)  is always even, hence in our case it is $0$ or $2$. 

We immediately prove a  result concerning the multiplicity.

\begin{proposition}  \label{P.finite}  Notation being as above, assuming that $\sigma$ is not torsion, the set  of $b\in B(\C)$ such that the multiplicity  at   $b$  of the Betti map for $\sigma$ is $>1$, is a  finite set.
\end{proposition} 

\begin{proof} Let us first deal with the points in a given compact region $K$ in $B(\C)$, assuming that $K$ is the closure of an open connected subset of $B$, where a determination of the Betti map $\beta$ is well-defined.   If $K$ contained infinitely many points of   multiplicity $>1$  for $\beta$, then $K$ would contain a whole real analytic arc of such points, 
i.e. where the differential $\d \beta$ vanishes. But then $\beta$ would be constant along this curve. Since the fibers of $\beta$ are analytic, $\beta$ then would be constant on $K$, and hence on $B$. By Manin's theorem, $\sigma$ would then be torsion. We conclude that $K$ contains only finitely many of the points in question. 

This suffices to cope with any portion of $B$ in the complement of any neighborhood of the boundary. It will suffice then to deal with any small disk around a boundary point $p_0\in \overline B-B$. We can cover this disk with finitely many triangular sectors, with a vertex in $p_0$, and such that $\beta$ is well-defined on each of the sectors. We now use that on each such sector $\beta$ is {\it definable} in the structure $\R_{an,exp}$.  This has been proved in \cite{JS} for the Legendre scheme, but then it holds also in our situation, on going to a Legendre model of our scheme (after suitable base-change). Then the derivatives of $\beta$ are also definable; this entails that the set of multiple points in a triangular sector is also definable, hence either is finite or contains a real-analytic arc; now the above argument again applies. 
\end{proof}

\begin{remark} If we deal with a scheme defined over $\overline\Q$, and if our  multiple points are also {\it torsion points} for $\sigma$ (which are those primarily of interest in this paper) then another argument is possible, avoiding the definability. Namely, one first remarks that  the degree of the field of definition of these points must tend to infinity, e.g. by Silverman's bounded height theorem. Then, again by bounded height, for a given point of `large' degree, its conjugates must fall in positive percentage outside a `small' neighborhood of the boundary points. This allows the argument in the first part of the lemma to apply. (See also \cite{Z} for another similar use of bounded height.)
\end{remark}

 \medskip
 
 Now we would like now to introduce a definition of multiplicity {\it at torsion points}  of $\sigma$, associated directly to the section $\sigma$ rather than to its Betti map.  
 
 \begin{definition}\label{D.multS} Let $b$ be a torsion point for $\sigma$ of order (dividing) $n$. We define the multiplicity of $\sigma$ at $b$, denoted $m_\sigma(b)$, as  the local multiplicity  of the intersection at $\sigma(b)\in \mathcal{E}$ between the curves $\sigma(B)$ and the torsion divisor $\mathcal{E}[n]$. 
 \end{definition}

{\it A priori} this depends as well on the $n$ such that $n\sigma(b)=O$. However the lemma which follows shows that  $m_\sigma(b)$ equals the previously considered multiplicity, and in particular does not depend on the said  integer $n$. 

\begin{lemma}\label{L.local}
Let $\sigma:B\to\E$ be a non-torsion section and  let $b\in T_\sigma$ be a torsion point for $\sigma$. Then the local multiplicity  $m_\sigma(b)$  
equals the multiplicity for the Betti map $\beta_\sigma$ at the point $b$.
\end{lemma}

\begin{proof} On multiplying by a nonzero integer the two multiplicities do not change (especially because multiplication on an elliptic curve is \'etale, which implies that on an elliptic surface this multiplication   is \'etale over the set of good reduction). 

So we reduce to the case when the torsion divisor intersected by $\sigma$ above $b$ is the image of the zero section.

Let then $t=x/y$ be  a local parameter at the (image of the) zero-section (so it is a local parameter at the origin of the fibers). If the section $\sigma$ has Weierstrass coordinates $(x_\sigma,y_\sigma)$ then define $t_\sigma=x_\sigma/y_\sigma$. 

Let now $b\in B$ be such that $\sigma(b)=0=0_b$. Then the zero-multiplicity  of $t_\sigma$ at the point $b$ equals the local intersection multiplicity of $\sigma(B)$ with the zero-section  (by definition). 

Let us denote by $m$ this multiplicity.

\medskip

We may now express the regular differential $\omega=\d x/y$ as a series 
\begin{equation*}
\omega=s(t) \d t,\qquad  s(t)=c_0+c_1t+...,
\end{equation*}
where the $c_i$ are functions on $B$ regular outside the bad reduction, and $c_0\neq 0$ outside the bad reduction, in particular at $b$. 

An  elliptic  logarithm $\tilde\sigma(u)$ of $\sigma(u)$ may be expressed, for $u$  in a suitable neighbourhood of $b,$ as 
\begin{equation*}
\tilde\sigma(u)=\int_{0_u}^{\sigma(u)}s_u(t)\d t.
\end{equation*}
Here the integral is on a path on the fiber $E_u$, where we may choose for instance the shortest path (recall that $u$ is supposed to lie near to the given point $b$, where $\sigma(b)=0_b$, hence $\sigma(u)$ is near to $0_u$).  

The path from $0_u$ to $\sigma(u)$ on $E_u$ corresponds, via the local parameter $t$ restricted to  $E_u$, to a path from $0\in\C$ to $t(\sigma(u))=t_\sigma(u)$. Thus we obtain 
\begin{equation*}
\tilde\sigma(u)=\int_{0}^{t_{\sigma(u)}} s_u(z)\d z.
\end{equation*}
Therefore, since $c_0(b)\neq 0$, 
\begin{equation*}
 |t_\sigma(u)| \ll_b |\tilde\sigma(u)|\ll_b|t_\sigma(u)|.
\end{equation*}
Hence, 
\begin{equation*}
 |\tilde\sigma(u)|  \asymp  d(u,b)^m,
\end{equation*}
where $d(.,.)$ is a distance function in a neighbourhood of $b$ in $B(\C)$.

On the other hand, by definition $\tilde\sigma(u) =\beta_1(u)\rho_1(u)+\beta_2(u)\rho_2(u)$. 

Conjugating this equation, we have the  vector equation 
\begin{equation*}
\begin{pmatrix}\tilde\sigma(u)\\  \  \\  \overline{\tilde\sigma(u)}\end{pmatrix} =\begin{pmatrix}\rho_1(u)   & \rho_2(u) \\ \  & \ \\  \overline{\rho_1(u)}   & \overline{\rho_2(u)}  \end{pmatrix} \begin{pmatrix}\beta_1(u)\\  \ \\\beta_2(u) \end{pmatrix}.
\end{equation*}
Now, the $2\times 2$-matrix is nonsingular at $u=b$ and hence is uniformly bounded together with its inverse in a whole neighbourhood of $b$. 

This shows that $||(\beta_1(u),\beta_2(u))||\asymp |\tilde\sigma(u)| \asymp  d(u,b)^m$, and now the fact that $\beta_1,\beta_2$ are real analytic, via a Taylor expansion, proves the lemma.

\end{proof}

\subsection{Multiplicity}  The above lemma gives two equivalent ways to define the multiplicity. In fact, the lemma proves even more, and we have the following list of equivalent definitions:
\smallskip
 
1 -  Local intersection between the image $\sigma(B)$ of the section and the torsion divisor  $\E[n]$ (as in the lemma). This notion is purely algebraic.

 2 -  Consider an  elliptic logarithm  $\tilde \sigma$ of the section; as a function from $B$ to $\C$,  it is well defined in a neighborhood of $b$ up to  periods. If $b\in B$ is a torsion point for $\sigma$ of order $n$, then  the function $\tilde\sigma -\omega/n$ will vanish at $b$ for a suitable choice of the period $\omega$ (well defined and holomorphic in a neighborhood of $b$).  We may define the multiplicity as the zero multiplicity of $\tilde\sigma -\omega/n$ at $b$.  (We can also reduce to the case $n=1$ by multiplication by $n$, as in the proof of the lemma.)


3 - Multiplicity of the Betti map at $b$, in the sense that both components are $0$ up to order $>m$  as real analytic functions in a neighborhood of $b$.   

4 - Multiplicity of  $t_{n\sigma}=x_{n\sigma}/y_{n\sigma}$ at $b$.  Note that $x/y$ is a  local parameter at the origin of the generic fiber. 
This may be expressed also in terms of the valuation $|.|_v$ associated to $b$, so that the value $-\log |t_{n\sigma}|_v$   is our multiplicity (provided $v$ is normalized so that its value group is $\mathbb{Z}$).

\medskip

\begin{theorem}\label{Thm:Theorem2.6}  There is a differential operator $\Xi$ of the second order on $B$, acting on local holomorphic functions, such that is $\sigma$ is a section and $\tilde{\sigma}$ a local determination of its logarithm, then $\Xi(\tilde{\sigma})=0$  if and only if $\sigma$ is torsion. 
 
 Let now $\sigma$ be non-torsion. 
 For a torsion point $b$ for $\sigma$,  
$m_\sigma(b)\le 2+\max{(0,\ord_b(\Xi(\tilde{\sigma})))}$. 
 In particular, this multiplicity does not exceed a certain explicitly computable function of $\hat h(\sigma)$. 

Also, there are only finitely many torsion points for $\sigma$ where  $m_\sigma(b)>1$. 
\end{theorem}

\begin{proof}  
We follow Manin \cite{M} (see also \S 6.3 of \cite{CZ}), and work on the Legendre scheme, as we may after a base change.    Let as above $\tilde\sigma$ denote a local determination of an elliptic logarithm of $\sigma$, and let $\Xi$ denote the usual Legendre-Gauss differential operator 
\begin{equation}\label{E.gauss}
\Xi=4\lambda(1-\lambda)\frac{\d^2}{\d\lambda^2}+4(1-2\lambda)\frac{\d}{\d\lambda}- 1,
\end{equation} 
where $\lambda$ is a suitable rational function on $B$.

Since $\Xi$ annihilates the periods, it follows that $\Xi(\tilde\sigma)$ is a well-defined function on $B$. By easy growth estimates on the coordinates of $\sigma$ we see that this function has no essential singularities on a complete model of $B$, and thus is a rational function on $B$  (which can be seen also directly: see equation (\ref{456PF}) below, which corrects (2) in Manin's paper). 

By condition 2. above on the multiplicity, this function  $\Xi(\tilde\sigma)$  vanishes at $b$ with multiplicity at least $-2$ + the multiplicity of $\sigma$ at $b$. 

This proves the first assertion, while the rest is a consequence of Proposition \ref{P.finite}.

\end{proof}


\begin{remark} Using Gabrielov Theorem (see \cite{BM}) one can give a uniform estimate for the cardinality of torsion points for $\sigma$, when $\sigma$ varies in a finitely generated group. However in the next theorem we shall achieve a still stronger result.

\end{remark} 

\begin{theorem} Let   $\Gamma$ be a finitely generated torsion-free group of sections $B\to \E$, defined over $\overline\Q$.   For  $\sigma\in\Gamma$, set  $M_\sigma:=\{b\in B : \sigma(b)   \ \hbox{is torsion},  m_\sigma(b)>1\}$.  Then the union $\bigcup_{\sigma\in\Gamma-\{0\}}M_\sigma$ is finite
\end{theorem}

\begin{proof} Let $\sigma_1,\ldots ,\sigma_r$ be a basis for $\Gamma$, and consider the respective Betti maps $\beta_1,\ldots ,\beta_r$, on some domain where they are well-defined.  (In those domains we make some definite choice of these maps, up to integers.) 

As in the previous proofs, we may cover $B$ with finitely many such domains (either compact or triangular regions with a vertex at a point in $\partial B=\overline B-B$ of bad reduction).

  A theorem of Andr\'e (see also \cite{CZ}) ensures
 that the $\beta_i$ are linearly independent over $\R$ modulo constant functions (since the sections $\sigma_i$ are linearly independent).
 
 We consider linear combinations $\sum x_i\beta_i$ with real $x_i\in\R$.   Among these linear combinations, those with rational coefficients $x_i$  correspond to  Betti maps (up to the addition of real constants) of elements in the division group of $\Gamma$. 
 
 \medskip
 
 Let us work now in one of the above domains, call it $D$.  We have a map from $\R^r\times D$  to $\mathrm{Mat}_2(\R)$, sending $(x_1,...,x_r,p)\mapsto \sum x_i\d\beta_i(p)$, where $\d\beta_i$ denotes the jacobian matrix of $\beta_i$ (with respect to some chosen coordinates in $D$).  This is a definable map, by the results that we have quoted  in the above proofs.  Therefore the set of  its zeros forms  a definable subvariety $Z$ of the domain.

We note that  if $\sigma$ is a section, corresponding to a rational point $\v=(c_1,\ldots ,c_r)\in\Q^r-\{0\}$,   the intersection of $Z$ with the fiber above $\v$ in $D$ is finite  by Proposition \ref{P.finite} (since the   sections are independent and $\v\neq 0$). However this fiber contains the set $M_\sigma$.  Also, we may restrict $\v$ to the closed unit cube by linearity. Now, by Gabrielov theorem, the number of connected components of the fibers is uniformly bounded, hence $|M_\sigma|$ is uniformly bounded as well.

Now, note that for a given $\sigma\in\Gamma$ the set $M_\sigma$ is stable by Galois conjugation over a number field of definition, hence by what has been proved above the degrees of the involved points are bounded independently of $\sigma$.  On the other hand, these points are torsion for $\sigma$, hence of bounded height by Silverman theorem, and the conclusion follows.
\end{proof}

\begin{remark} We have argued for sections defined over $\overline\Q$. However the result holds true for any ground field of characteristic $0$. A proof comes from specialisation, as in the paper \cite{CMZ}, although here the specialization argument  is  easier then in \cite{CMZ}. We give here a sketch of the argument.

We can always reduce to a covering of the Legendre scheme $\mathcal{E}\to B$, where $B\to \P_1-\{0,1,\infty\}$   is defined over $\bar\Q$, but we consider now a finitely generated group of sections which are not necessarily defined over a transcendental extension  $K$ of   over $\overline{\Q}$. Suppose for simplicity that $K$ has transcendence degree $1$ over $\bar{\Q}$, so it is a finite  algebraic extension of $\bar{\Q}(t)$.  Geometrically, the field $K(\lambda)$ corresponds to a surface $S$ defined over $\bar{\Q}$, endowed with a projection to $B$, on which $\lambda$ is a rational function. The function field of $S$ is an algebraic extension of $\bar{\Q}(t,\lambda)$ for some rational function $t\in\bar{\Q}(S)$. The elliptic scheme can be viewed as a scheme over $S$, and $S$ can be taken to be affine and such that  the `bad reduction' is confined to its   points at infinity.
A torsion point $b\in B$ for a section $\sigma: B\to \mathcal{E}$ is necessarily algebraic over $K$; hanece it can be viewed in $S$ by definining $\lambda$ and the $x$-coordinate (and consequently the $y$-coordinate) by a certain  algebraic function of $t$. Geometrically, this corresponds to an algebraic relation $\varphi_b(x,\lambda)=0$, which in turn provides an algebraic curve in $S$; we call it a torsion curve. 

The idea now is to cut the surface $S$ with a `generic' curve $Y$ defined over $\bar{\Q}$ and to consider the elliptic scheme over such a curve. The finiteness result just proved over $\bar{\Q}$ should imply the corresponding finiteness statement for our original sections, which are not defined over $\bar{\Q}$.   
The problem, which was the main obstacle in the specialization procedure carried out in \cite{CMZ}, is that such a curve $Y$ might avoid all but finitely many torsion curves; hence any finiteness result for the intersection of $Y$ with the torsion curves of a certain type (namely those with $m_\sigma(b)>1$) would be meaningless. As noted in \cite{CMZ}, on any affine surface $S$ defined over $\bar{Q}$ one can construct a sequence of algebraic curves $B_1,B_2,\ldots$ such that each algebraic curve $Y\subset S$ defined over $\bar{\Q}$ avoids all but finitely many of them (in the sense that it will intersect them only at infinity). However, such a sequence of curves $B_1,B_2,\ldots$ would necessarily have a degree tending to infinity. This is the content of the following claim:

\smallskip
{ \tt Claim}. {\it Let $\bar{S}$ be a projective algebraic surface defined over $\bar{\Q}$, $S\subset \bar{S}$ an affine subset. Let $H=\bar{S}-S$ be the divisor at infinity of $S$. Let $B_1,B_2,\ldots$ be a sequence of irreducible curves in $S$ such that the intersection product $\bar{B}_i\cdot H$ of their closure  with $H$ is uniformely bounded. 
Then there exists a curve $Y\subset S$ defined over $\bar{\Q}$ such that $Y$ intersects infinitely many of the curves $B_i$.
}

\smallskip

{ \tt Proof of the Claim}. Let $d\geq 1$ be an upper  bound for the intersection product $B_i\cdot H$ and 
choose $d+1$ curves $Y_1,\ldots, Y_{d+1}$ in $S$  such that for every pair $(j,h)$ with $1\leq j<h\leq d+1$, the complete curves $\bar{Y}_{j}, \bar{Y}_h$ do not intersect at infinity. 
Then for each $i$, the complete curve $\bar{B}_i$ can intersect at most $d$ of the $\bar{Y}_j$ at infinity, so it must intersect at least one of them in $S$. It follows that for at least one index $j$, the curve $Y_j$ intersects infinitely many of the $B_i$.
\end{remark}

Now, in our situation, we claim that  the torsion curves in question are given by equations of the form $\varphi_b(x,\lambda)=0$, where both partial degrees of the polynomials $\varphi_b$ are bounded. The degree in $x$ corresponds to the functional height of $\lambda$, which is bounded by a functional version of Silverman's theorem since the point is torsion. Reciprocally, the degree in $\lambda$ is bounded as before by an application of Gabrielov's theorem; this time we have to use the fact that our point is not only torsion, but satisfies $m_\sigma(b)>1$. 

It follows that the torsion curves we have to examine have bounded degree with respect to the divisors at infinity $H=\bar{S}-S$, so we can apply the previous claim and conclude that for one curve $Y$ in $S$, $Y$ intersects infinitely many torsion curves $b(x,\lambda)=0$ with $m_\sigma(b)>1$, and the theorem already proved gives the desired contradiction.

\section{Canonical height and Betti map}\label{Sec:Legendre}

Throughout this section, $\pi:E \rightarrow \P_1$ denotes the Legendre scheme, i.e. the elliptic surface associated to the following elliptic curve, defined over the function field $\C(\lambda)$ of $\P_1$:
\begin{equation}\label{Eq:Legendre}
	y^2=x(x-1)(x-\lambda).
\end{equation}

Let $S=\{0,1,\infty\}$ be the set of points of bad reduction for $\pi$. In the future we will also use the notation $\mathcal{S}=\pi^{-1}(S)$. 

	In this note, we will frequently work with the abelian logarithm of an elliptic curve $X/\C$. We recall that, for this to be defined, a lattice $\Lambda_X \subset \C$, such that $\C/\Lambda_X \cong X$ has to be chosen. 
	
	\begin{remark}\label{Rmk:weierstrass}
			When a Weierstrass form $y^2=x^3+ux+v$ for $X$ is chosen, there exists one unique such lattice $\Lambda_X$ that satisfies the extra condition $g_2(\Lambda_X)=-4u, g_3(\Lambda_X)=-4v$. When we work with a specific Legendre equation as \eqref{Eq:Legendre}, we obtain a corresponding  Weierstrass equation as above after the substitution $x\mapsto x+(1+\lambda)/3$, and   assume that the above canonical choice of a lattice for the abelian logarithm has been made.
	\end{remark}

\subsection{Height as an integral and consequences}\label{Hic}

The following is the main result of this section. The result was not present in the literature in the following form, though it can be deduced from the work of DeMarco and Mavraki \cite{DMM} using Wirtinger's formula: we sketch this argument at the end of the following section. However, the proof presented here seems to be much more elementary in nature.


\begin{theorem}\label{Prop:equality}
	Let $r:B \rightarrow \P_1$ be a finite morphism, and $B$ be a smooth complete complex curve. Let $\sigma:B \rightarrow E_B := E\times_{(\pi,r)}B$ be an algebraic section of $\pi_B:E_B \rightarrow B$.  Then, the following equality holds:
	\begin{equation}\label{soughtequality}
	\hat{h}(\sigma) = \int_{B \setminus r^{-1}(S)} \sigma^*(\d \beta_1 \wedge \d \beta_2),
	\end{equation}
	where $(\beta_1,\beta_2)$ is a (local branch) of the Betti map on $E$. 
\end{theorem}

Moreover we will see, as part of the proof of Theorem \ref{Prop:equality}, that the set $T_{\sigma}$ introduced at the beginning of this paper is distributed as the measure associated to the $2$-form $\sigma^*(\d \beta_1 \wedge \d \beta_2)$.

We immediately notice the following corollary.

\begin{corollary}\label{Cor:rational_value}
	Let $\sigma:B \rightarrow E_B$ denote a section of the elliptic surface $E_B \rightarrow B$. Then the integral $\int_{B \setminus r^{-1}(S)} \sigma^*(\d \beta_1 \wedge \d \beta_2)$ has a rational value.
\end{corollary}
\begin{proof}
	We have by \cite[Section 11.8]{ellipticsurfaces} that $\hat{h}(\sigma) \in \Q$. Hence, Corollary \ref{Cor:rational_value} is an immediate consequence of Theorem \ref{Prop:equality}.
\end{proof}

Before coming to the proof of Theorem \ref{Prop:equality}, we give here another expression for the $(1,1)$-form $\d \beta_1 \wedge \d \beta_2 \in \Omega^2(E \setminus \mathcal{S})$ that might be more suitable for calculations. To fix some notation, let $\Lambda_{\lambda}$ be a local choice of a lattice in $\C$ corresponding to the elliptic curve $E_{\lambda}$ with Weierstrass form $y^2=x(x-1)(x-\lambda)$ (see Remark \ref{Rmk:weierstrass}), and let $\rho_1(\lambda),\rho_2(\lambda) \in \C$ be a (continuous) local choice of periods. Moreover, let $d(\lambda):=\rho_1(\lambda)\overline{\rho_2(\lambda)}-\rho_2(\lambda)\overline{\rho_1(\lambda)}=2iV(\lambda)$, where $V(\lambda)$ denotes the (oriented) area of the fundamental domain of $\Lambda_{\lambda}$. Let $\eta_{\lambda}(\zeta)$ denote the linear extension to $\C$ of the quasi-period function $\eta_{\lambda}$ (which is defined on the lattice $\Lambda_{\lambda}$ as in \cite[VI.3.1]{silverman1994advanced}), and let $\eta_i(\lambda) =\eta_{\lambda}(\rho_i(\lambda)), \ i=1,2$. Then, the following equality, which is proven in Remark \ref{Rmk:calculationforT} below, holds (here we are denoting, with a slight abuse of notation, by $\beta_i := \beta_i(\lambda)$ the Betti coordinates $\beta_i(\sigma(\lambda))$):
\begin{flalign}\label{riscrittura}
\begin{split}
&\d \beta_1 \wedge \d \beta_2 \\
&=\frac{1}{2iV(\lambda)}\left[\frac{1}{2iV(\lambda)}((\eta_1 \bar{\rho_2}-\eta_2\bar{\rho_1})z+2\pi i\bar{z}){\frac{\d \lambda}{2 \lambda}}+\d z\right]\wedge\left[\frac{1}{2iV(\lambda)}(2\pi iz+(-\bar{\eta_1}\rho_2+\bar{\eta_2}\rho_1)\bar{z}){\frac{\d \bar{\lambda}}{2 \bar{\lambda}}}+\d \bar{z}\right],
\end{split}
\end{flalign}

where $z:= z(\lambda):= \log^ab(\sigma(\lambda))$.


%

\smallskip

\begin{example}\label{explicitexample}
	Let us do an explicit computation of the terms of (\ref{soughtequality}), in the case of a specific section, for instance:
	\[
	\sigma(\lambda)=(2,\sqrt{2(2-\lambda)}).
	\]
	This section is not defined over the base curve $\P_1$. It is, however, well defined as a section of the elliptic surface $E_B := E \times_{(\pi,\phi)}B \rightarrow B$, where $B \cong \P_1$, and $\phi(t)=t^2+2$. 
	
	One can compute $\hat{h}(\sigma)$ explicitly, by using the intersection product on a proper regular model of $E_B$ (see \cite[Section 11.8]{ellipticsurfaces})\footnote{Our normalization of the height function differs by a multiplicative factor of $\frac{1}{2}$ from that of \cite{ellipticsurfaces}.}. A simple application of Tate's algorithm reveals that the fibration $\pi:E_B \rightarrow B$ has $5$ singular fibers, four of which are of type $I_2$ (the ones over $\lambda= 0,1$) and one of type $I_4$ (the one over $\lambda=\infty$). Looking at the intersection of the section $\sigma$ with the singular fibers reveals that the canonical height $\hat{h}(\sigma)=\frac{1}{2}$.
	
	To explicit the calculation we will think of $B$ as being a union of $2$ projective plane charts (which we denote by $B_+, B_-$) that are united through the segment $[2,\infty]$, with the topology that switches chart when ``we cross the segment''. We use $\lambda$ as a parameter on the base $B$.
	
	The $(1,1)$-form $\d \beta_1(\lambda)\wedge \d \beta_2(\lambda)=\sigma^*(\d \beta_1 \wedge \d \beta_2) \in \Omega^2(E_B \setminus S_B)$ is equal to (\ref{riscrittura}) above, with 
%
%
	\begin{equation}
	z=z(\lambda)=\int_{\infty}^{2}\d x/(\pm \sqrt{x(x-1)(x-\lambda)}),
	\end{equation}
	
	where the determination of the square root changes when we change the chart.

	Hence, we deduce the following integral identity from Theorem \ref{Prop:equality}:
	\[
		\frac{1}{4}=\frac{1}{2}\cdot \hat{h}(\sigma)=\frac{1}{2}	\cdot \int_{B} \d (\beta_1\circ \sigma) \wedge \d (\beta_2\circ \sigma)=  \int_{\lambda \in \C} \d \beta_1(\lambda) \wedge \d \beta_2(\lambda).
	\]
%
\end{example}

\subsection{Proof of Theorem \ref{Prop:equality}}

We recall the following Lemma, which follows directly from \cite[Theorem 1.3]{BW}, and that will be used in the proof:
\begin{lemma}\label{ominimalmeasure}
	Let $X \subset \R^k$ be a bounded subset of $\R^k$ definable in an o-minimal structure. Let, for each $n>0$,
	\[
	\alpha_n := \#\left\{\left(\frac{a_1}{n},\dots,\frac{a_k}{n} \right)\in X: \ a_1,\dots, a_k \in \Z \right\}.
	\]
	Then, the limit $\lim_{n \to \infty} n^{-k}\alpha_n$ exists and is equal to $\mu(X)$, the Lebesgue measure of $X$ (which exists).
\end{lemma}

\begin{proof}[Proof of Theorem \ref{Prop:equality}]
	Let us first consider the case where $\sigma$ is torsion. Since the Betti map of a torsion section is constant, in this case, left and right hand side of (\ref{soughtequality}) are both equal to $0$. 
	
	We restrict now to the case where $\sigma$ is not torsion, and consider, for each $n \geq 1$, the following quantity:
	\begin{equation}\label{targetorsion}
		A_n := \# \{t \in B(\C)\setminus r^{-1}(S): \sigma(t) \text{ is }n-\text{torsion in } E_t\}.
	\end{equation}
	We notice that, since $\sigma$ is not torsion, this quantity is finite for each $n \geq 1$.
	We claim the following (which obviously implies the thesis):
	\begin{enumerate}
		\item[(a)] The limit $\lim_{n \to \infty}\frac{A_n}{n^2}$ exists, is finite, and $\lim_{n \to \infty}\frac{A_n}{n^2}=\hat{h}(\sigma)$;
		\item[(b)] The limit $\lim_{n \to \infty}\frac{A_n}{n^2}=\int_{B \setminus r^{-1}(S)} \sigma^*(\d \beta_1 \wedge \d \beta_2)$.
	\end{enumerate}
	Let us first prove $(a)$. We know that (see e.g. \cite[Section III.9]{silverman1994advanced} or \cite[Sections 2,3]{SerreMW}):
	\[
	\hat{h}(\sigma) =\lim_{n \to \infty} \frac{<n \sigma, O>}{n^2},
	\] 
	where $n\sigma$ denotes, with a slight abuse of notation, the graph of the section $n \sigma$, $O$ denotes the zero section of $\pi_B$, and $<n \sigma, O>$ denotes the intersection product in a smooth proper model of $\widetilde{E_B}$. We write:
	\[
	<n \sigma, O>=A_n+\delta_n+s_n,
	\]
	where $s_n$ denotes the intersection of $n \sigma$ and $O$ on the singular fibers of $\pi_B:\widetilde{E_B} \rightarrow B$, and $\delta_n$ is a correction term that keeps track of the intersection that happens with multiplicity greater than $1$. I.e.:
	\[
	\delta_n=\sum_{t \in B\setminus r^{-1}(S)} (\ord_t(n\sigma)-1),
	\]
	where $\ord_t(n\sigma)$ denotes the multiplicity of intersection of $n\sigma$ and $O$ at $t$. 
	We will now prove that $\delta_n+s_n=O(1)$, as $n \to \infty$. 
	
	We prove first that $s_n=O(1)$. To do so, it suffices to show that, for a point $s \in B$ of bad reduction for $\pi_B$, $<n\sigma,O>_s$ (i.e. the local multiplicity of intersection) is bounded in $n>0$. Let $s \in r^{-1}(S)$, and let $m \in \Z$ be the least positive integer such that $(m\sigma)(s)=0$ on the singular fiber $E_s:=\pi_B^{-1}(s)$. Since $<n \sigma, O>_s$ can be positive only when $n$ is a multiple of $m$, we may assume without loss of generality (replacing, if necessary, $\sigma$ by $m\sigma$), that $m=1$, i.e. that $\sigma(s)=0$. 
	
	Let now $\lambda$ denote a local parameter for $s \in B$. Let $K$ denote the completed field of Laurent series $\C\{\{\lambda\}\}$, and let $O_K$ denote the ring of integers of $K$. Let:
	\[
	E_1(K):= \{P \in E(K):P(s)=0\}.
	\]
	
	We recall that $E_1(K)$ has the structure of a Lie group over the local field $K$, given by the restriction of the sum operation on the elliptic curve $E/K$. The Lie group $E_1(K)$ is isomorphic to a Lie group $(\lambda O_K,\tilde{+})$, in a way that we briefly recall now.
	
	We recall from \cite[Chapter IV]{silverman} that the formal group on the elliptic curve $E/K$ is the unique formal power series:
	\[
	F(+)(X,Y) \in K[[z_1,z_2]],
	\]
	such that
	\[
	F(+)(z(P),z(Q))=z(P+Q), \quad P, Q \in E
	\]
	as \textit{formal} power series, where $z(P):= x(P)/y(P)$, denotes the $z$-coordinate in a given minimal Weierstrass model for $E/K$. 
	
	We recall, moreover, that the following holds:
	\[
	F(+)(z_1,z_2)=z_1+z_2+O(z^2).
	\]
	
	One can verify that, when an integral Weierstrass model is chosen, $F(+) \in O_K[[z_1,z_2]]$. Hence, $F(+)(X,Y)$ converges when $X, Y \in \lambda O_K$. One can then see that the Lie group $E_1(K)$ is isomorphic to the Lie group $(\lambda O_K,\tilde{+})$, where $\tilde{+}(z_1,z_2):= F(+)(z_1,z_2)$ (see e.g. \cite[Proposition VII.2.2]{silverman}).
	
	It follows that, for $P \in E(K)$, denoting by $z(P)$ the $z$-coordinate in a minimal Weierstrass model, one has that:
	\begin{equation}\label{multiplesorder}
	z(n\sigma)=nz(\sigma)+O(z^2), \quad \forall n\in \Z.
	\end{equation}
	
	Moreover, for any $P \in E_1(K)$, we have that $\ord_s(P)=\ord_{\lambda}(z(P))$ (since $z$ is a local parameter for the $0$-section $O$). Hence, we have that $\ord_s(n\sigma)=\ord_{\lambda}(z(n\sigma))=\ord_{\lambda}(z(\sigma))=\ord_s(\sigma)$, where the middle equality follows from (\ref{multiplesorder}). Hence we have that $s_n=O(1)$.

	We prove now that $\delta_n$ is uniformly bounded (for any $n>0$). We notice that Lemma \ref{L.local} implies that, in order for the intersection of $n \sigma$ and $O$ to be of order greater than $1$ at a certain point $t \in B$, the differential of (a local branch of) the Betti map $(\beta_1,\beta_2)\circ \sigma$ would have to be $0$ at the point $t$. Because of Proposition \ref{P.finite}, this happens only for finitely many base points $t$ in the base. We denote them by $t_1,\dots, t_k  \in B$. 
	
	For each $i=1,\dots, k$, let $\lambda_i$ denote a uniformizer for $t_i \in B(\C)$, and let $\rho^i_1, \rho^i_2$ denote a local choice (in a neighbourhood of $t_i$) of periods for the elliptic logarithm (see e.g. \cite[Section 1.1]{CMZ}). We then have then that (by standard intersection theory on complex surfaces, see e.g. \cite[Chapter I]{beauville}):
	\begin{equation}\label{multiplicity}
	\ord_{t_i}(n\sigma)=\dim_{\C} \frac{\C\{\lambda_i\}}{(n\tilde{\sigma}(\lambda_i)-nb_1\rho^i_1(\lambda_i)-nb_2\rho^i_2(\lambda_i))}, 
	\end{equation}
	where $\C\{\lambda_i\}$ denotes the ring of locally analytic functions in the variable $\lambda_i$, $\tilde{\sigma}$ denotes the abelian logarithm of $\sigma$, and $b_1, b_2 \in \C$ are defined through the following condition: \[\tilde{\sigma}(t_i)=b_1\rho^i_1(t_i)+b_2\rho^i_2(t_i).\]	
	
	Since the ideal $(n\tilde{\sigma}(\lambda_i)-nb_1\rho^i_1(\lambda_i)-nb_2\rho^i_2(\lambda_i)) \subset \C\{\lambda\}$ does not depend on $n>0$, we have that the right hand side of (\ref{multiplicity}) does not depend on $n>0$, and we denote this quantity by $R_i$. As an immediate consequence of the above argument, we have that:
	\[
	\delta_n \leq R_1 + \dots + R_k - k,
	\] 
	from which we conclude that $\delta_n =O(1)=o(n^2)$, thus proving point $(a)$.
	
	We prove now point $(b)$. We use the fact that the Betti map is (locally) definable and bounded (by definable, we will always mean definable in $(\R_{an,exp})$). 
	
	Namely, it is proven in \cite[Section 10]{JS} that there exists a (definable) partition $B\setminus r^{-1}(S)=Y_1 \cup \dots Y_l$, and, on each $Y_j$ there exists a well defined branch of the Betti map, which we will denote by $B^j=(\beta^j_1,\beta^j_2)\circ \sigma: Y_j \rightarrow \R^2$, such that $B^j$ is definable, and in \cite[Proposition 4]{JS} that it is bounded. 
	
	Let us consider now, for each $j=1,\dots, l$, the following definable set in $\R^2 \times \R^2$:
	\[
	X_j:= \{(y,t) \in \R^2 \times \R^2: t \in Y_j \text{ and } y=B_j(t)\},
	\]
	i.e. $X_j$ is the transpose of the graph of $B^j$. Applying Hardt's Theorem (\cite[Theorem 9.1.2]{tametopology}) to $X_j$, we know that there exists a finite partition of $\R^2$, say $A_j^1\cup\dots \cup A_j^{d(j)}=\R^2$, such that $X_j$ is definably trivial over each $A_j^m$.  We recall that this means that, for each $m \leq d(j)$, there exists a definable set $F_j^m$, and a (definable) isomorphism $h_{A_j^m}:X_j\cap (A_j^m\times \R^2)\rightarrow A_j^m\times F_j^m$, that commutes with the projection to $A_j^m$. Moreover, since the Betti map is bounded on each $Y_j$, we may assume that, for each $m \leq d(j)$, either $A_j^m$ is bounded, or $X_j\cap (A_j^m\times \R^2)=\emptyset$. 
	
	We define now, for each $j=1,\dots, l$, $m \leq d(j)$:
	\begin{equation}
		f_j^m:= \#F_j^m
	\end{equation}
	
	We notice that $f_i^j$ is always finite. In fact, this is a direct consequence of the fact that the fibers of the Betti map $B_j$ are isolated points (see \cite[Proposition 1.1]{CMZ}).
	Then, the following equality holds:
	\begin{equation}\label{measure}
	\sum_{m \leq d(j)} f_{j}^m\mu(A_j^m)=\int_{B\setminus r^{-1}(S)} \sigma^*(\d \beta_1 \wedge \d \beta_2).
	\end{equation}
	
	In fact, notice that both left and right hand side of (\ref{measure}) are equal to the measure of the graph of the section $\sigma:\P_1(\C)\setminus S\rightarrow E\setminus \mathcal{S}_B$ 
	, given by the integration of the restriction of the $(1,1)$-form $\d \beta_1 \wedge \d \beta_2 \in \Omega^2{E \setminus \mathcal{S}_B}$.
	
	We notice now that, if $t \in \P_1(\C)\setminus S$, $\sigma(t) \in E_t$ is $n$-torsion if and only if the Betti coordinates $(\beta_1(\sigma(t)),\beta_2(\sigma(t)))$ are rational with denominator dividing $n$. Hence, the following equality holds, for each $n>0$:
	\begin{equation}\label{counting}
		A_n=\sum_{m \leq d(j)}f_{j}^m\alpha_{n,j}^m,
	\end{equation}
	where:
	\[
	\alpha_{n,j}^m:= \left\{\left(\frac{a}{n},\frac{b}{n}\right)\in A_j^m: \ a,b \in \Z\right\}.
	\]
	Hence:
	\[
		\frac{A_n}{n^2}=\sum_{m \leq d(j)}\frac{\alpha_{n,j}^m}{n^2},
	\]
	Letting $n \to \infty$, the thesis follows from (\ref{measure}), (\ref{counting}) and Lemma \ref{ominimalmeasure}.
\end{proof}

\begin{remark}\label{Rmk:torsion}
	In the proof of Theorem \ref{Prop:equality}, we treated separately the case of a torsion section $\sigma$. This was done because, in this case, the number $A_n$ (defined in (\ref{targetorsion})) is not defined for all $n$. However, when $\sigma$ has order $m \neq 1$ (i.e. $\sigma \neq O$), the number $A_n$ is still defined for all $n$ coprime to $m$. Therefore, for such $\sigma$, one could use the same argument that works for a non-torsion section, changing the limits over $n$ into limits over $n$ coprime to $m$.
\end{remark}

\begin{remark}
	To prove Theorem \ref{Prop:equality}, one may also use a different argument. Namely, one first proves equality (\ref{soughtequality}) when $r:B \rightarrow \P_1$ is defined over $\bar{\Q}$ using a height argument, and then uses a continuity argument to prove it for every $r$. We give here a sketch of these two steps.

	
	Assume that $r:B \rightarrow \P_1$ is defined over a number field $K$. One can prove in a straightforward way that: 
	\begin{equation}\label{Eq:approxsoughtequality}
	\lim_{n \to \infty}\frac{A^{\epsilon}_n}{n^2}=\int_{B \setminus r^{-1}(S_{\epsilon})} \sigma^*(\d \beta_1 \wedge \d \beta_2),
	\end{equation}
	where $S_{\epsilon}$ is a neighborhood of radius\footnote{With respect to some choice of a metric, which is irrelevant for our purposes. For instance, one may choose the Fubini-Study metric on $\P_1(\C)$.} $\epsilon$ of $S$, and 
	\[
	A^{\epsilon}_n:=\# \{t \in B(\C)\setminus r^{-1}(S_{\epsilon}): \sigma(t) \text{ is }n-\text{torsion in } E_t\}.
	\]
	
	Now, one proves that:
	
	\begin{equation}\label{Bound:Uniform}
		\left| \frac{A^{\epsilon}_n}{A_n}-1 \right|\leq \frac{C}{|\log \epsilon|}, \text{ where } C \in \R_+.
	\end{equation}

	Partitioning $A_n$ in Galois orbits, this becomes a consequence of the fact that the points $t \in \cup_{n \in \N}A_n$ are of bounded height in $\P_1(\bar{\Q})$ \footnote{This is a direct consequence of the well-known result of Tate \cite{Tate} that, if $\pi:E\rightarrow B$ is an elliptic fibration, and $P:B \rightarrow E$ is a non-torsion section, then the function $t \mapsto \hat{h}_{E_t}(P_t)$, is, up to a bounded constant, a Weil height on $B$.}, 
	and the fact\footnote{One way to prove this fact is to use the following easy result, which can be found, for instance, in \cite[Remark 3.10(ii)]{LectureNotesZannier}. Let $\xi \in L$, where $L\subset \C$ is a Galois number field of degree $d$ over $\Q$, then:
	\[
	\sum_{g \in \operatorname{Gal}(L/\Q)}|\log|\xi^g||\leq 2d\hat{h}(\xi).
	\]} that points $x \in \P_1(\bar{\Q})$ of height $\leq C'$ have at most $\frac{C'}{|\log \epsilon|}[K(x):K][K:\Q]$ conjugates of absolute value $< \epsilon$.
	
	By letting $\epsilon \to 0$ in (\ref{Eq:approxsoughtequality}) and using (\ref{Bound:Uniform}) to employ a uniform convergence argument, one gets the sought equality (\ref{soughtequality}) for $r$.
	
	When $r$ is not defined over a number field, let $K_0 \subset \C$ denote its minimal field of definition, and let $K=K_0\bar{\Q} \subset \C$. The field $K$ has finite transcendence degree, hence there exists an integral algebraic variety $X/\bar{\Q}$ such that $\bar{\Q}(X)\cong K$. By construction, up to restricting $X$ to a Zariski open subset, we may assume that there exists an algebraic family of morphisms
	\[
	r_t:B_t \rightarrow \P_1, \ t \in X,
	\]
	and sections:
	\[
	\sigma_t:B_t \rightarrow E\times_{(\pi,r_t)}B_t, \ t \in X,
	\]
	such that there exists $t_0 \in X(\C)$, such that $B_{t_0}=B$, $r_{t_0}=r$ and $\sigma_t=\sigma_{t_0}$. 
	
	We want to show now that, for any $t \in X(\C)$ (and hence, in particular, for $t=t_0$), the following equality holds:
	\begin{equation}\label{soughtequalityt}
		\hat{h}(\sigma_t) = \int_{B_t \setminus r_t^{-1}(S)} \sigma^*(\d \beta_1^t \wedge \d \beta_2^t),
	\end{equation}
	
	where, for a point $P \in B_t$, $\beta_1^t$, $\beta_2^t$ denote the Betti coordinates of $\sigma_t(P)$.
	Now, we know that, for each $t \in X(\bar{\Q})$, (\ref{soughtequalityt}) is true by the argument presented above in this Remark. Since $X(\bar{\Q})$ is dense (in the euclidean topology) in $X(\C)$, it is hence sufficient to show that both right and left hand side of (\ref{soughtequalityt}) are continuous in $t \in X(\C)$. 
	
	For the left hand side, this is an immediate consequence of the fact that the height may be expressed through an explicit intersection formula on a smooth proper model (see \cite[Section 11.8]{ellipticsurfaces}), hence a standard flatness argument tells us that it is constant (hence, continuous) for $t$ in a nonempty Zariski-open subset of $X$ (that automatically contains $t_0$ since this is, by construction, a generic point).
	
	For the right hand side, this follows from a dominated convergence argument, which proceeds as follows. Since the $(1,1)$-form $\d \beta_1 \wedge \d \beta_2$ diverges only when $\lambda$ approaches $0,1$ or $\infty$, it is sufficient to give a bound for $\d \beta_1 \wedge \d \beta_2$ in small disks around these three points, locally uniformly in $t \in X(\C)$. Without loss of generality, one may do so only for a small disk around $0$. Here, one may use results of Jones and Schmidt \cite{JS} to obtain the following bound in a circle $|\lambda| < \epsilon$, locally uniformly in $t \in X(\C)$:
	\begin{equation}\label{bound}
	\d \beta_1^t \wedge \d \beta_2^t= \O\left(\frac{1}{|\lambda|^2|\log \lambda|^4}\right)\d \lambda \wedge \d \bar{\lambda}. 
	\end{equation}
	Since the right hand side of (\ref{bound}) is an integrable $2$-form in the circle $|\lambda|<\epsilon$, this allows a dominated convergence argument to be employed to show that the right hand side of (\ref{soughtequalityt}) is continuous, hence reproving Theorem \ref{Prop:equality}.
\end{remark}

%
%

\section{Addendum: Comparison with a measure coming from dynamics}


In this addendum we compare the measure $\sigma^*(\d \beta_1 \wedge \d \beta_2)$ with a measure appearing in the following theorem of DeMarco and Mavraki \cite[Section 3]{DMM}:

\begin{theorem}[DeMarco, Mavraki]\label{Theo:demarcomavraki}
	Let $\pi:E \rightarrow B$ be an elliptic surface and $P : B \rightarrow E$ a non-torsion section, both defined over $\Q$. Let $S \subset E$ be the union of the finitely many singular fibers in $E$. There is a positive, closed $(1,1)$-current $T$ on $E(\C) \setminus S$ with locally continuous potentials such that ${T}|_{E_t}$ is the Haar measure on each smooth fiber, and $P^*T$ is equal to a measure $\mu_{P}$, that satisfy the following property. For any infinite non-repeating sequence of $t_n \in B(\bar{\Q})$, such that $\hat{h}_{E_{t_n}}(P_{t_n}) \rightarrow 0$ as $n \to \infty$, the discrete measures
	\[
	\frac{1}{\#\operatorname{Gal}(\bar{\Q}/\Q)t_n}\sum_{t \in \operatorname{Gal}(\bar{\Q}/\Q)t_n}\delta_{t_n}
	\]
	converge weakly on $B(\C)$ to $\mu_P$. \footnote{DeMarco and Mavraki \cite{DMM} proved this result also in the non-archimedean setting.}
\end{theorem}

In particular, we will prove that the current $T$ is equal to the $(1,1)$-current defined by the $(1,1)$-form $\d \beta_1 \wedge \d \beta_2$ on $E \setminus \mathcal{S}$ (we warn the reader that, whereas in the last section the form $\d \beta_1(\lambda) \wedge \d \beta_2(\lambda)$ we worked with was defined on the basis of the elliptic fibration $B$, the $(1,1)$-form $\d \beta_1 \wedge \d \beta_2$ we are working with now is define on the whole fibration space $E \setminus \mathcal{S}$). A direct consequence of this is that the pullback measure $\sigma^*(\d \beta_1 \wedge \d \beta_2)$ is characterized by the arithmetic condition appearing in Theorem \ref{Theo:demarcomavraki}.

The restriction of the current $T$ to the open set $E \setminus (\mathcal{S} \cup O)$ (where we are denoting by $O$, with a slight abuse of notation, the image of the zero section of $\pi:E \rightarrow B$), can be written down explicitly as follows (\cite[Section 3.3]{DMM}):
\begin{equation}\label{Expr:11current}
	T=\frac{1}{2\pi i}\ \d \d^c {H_N},
\end{equation}
where $H_N$ denotes the N\'eron local (archimedean) height function. We recall that the following formula holds \cite[p. 466]{silverman1994advanced}:
\begin{equation}\label{Expr:altezzalocale}
	H_N=-\log |e^{-\frac{1}{2}z\eta_{\lambda}(z)}\sigma_{\lambda}(z)\Delta(\Lambda_{\lambda})^{\frac{1}{12}}|.
\end{equation}
Here $\Lambda_{\lambda}$ denotes a lattice in $\C$ such that $\C/\Lambda_{\lambda} \cong E_{\lambda} := \pi^{-1}(\lambda)$, $z$ denotes the complex variable of $\C/\Lambda_{\lambda}$ and $\eta_{\lambda}$ and $\sigma_{\lambda}$ indicate the semiperiod function $\eta$ and, resp., the function $\sigma$ (as defined in \cite[VI.3.1,I.5.4]{silverman1994advanced}) associated to the lattice $\Lambda_{\lambda}$. Since $\sigma_{\lambda}(z)\Delta(\Lambda_{\lambda})^{\frac{1}{12}}$ is a holomorphic function in both variables $\lambda$ and $z$, this gives the following expression for $T$:
\begin{equation}\label{Expr:11current2}
	T=\frac{1}{4\pi i}\ \d \d^c(\Re (z\eta_{\lambda}(z))).
\end{equation}



We shall prove that the current $T$ matches the $2$-form $\d \beta_1 \wedge \d \beta_2$ in two different ways: through a direct calculation (of which we give a sketch in Remark \ref{Rmk:calculationforT}) and through a dynamical argument (in Corollary \ref{Cor:dynamicalforT}).

We notice that both $T$ and $\d \beta_1 \wedge \d \beta_2$ restrict to the Haar measure on the fibers, normalized in such a way that the area of each fiber is $1$. 

\begin{remark}\label{Rmk:calculationforT}
	Let us give a sketch of a calculation that shows that $T= \d \beta_1 \wedge \d \beta_2$. This proves also \eqref{riscrittura}.
	One first checks (using (\ref{Expr:11current2})) that:
	\begin{equation}\label{Expr:Traw}
	T=\frac{1}{4\pi i}\partial \bar{\partial} \left[ \begin{pmatrix}
	z \\ \bar{z}
	\end{pmatrix}^\intercal
	\begin{pmatrix}
	\eta_1 & \eta_2 \\
	\bar{\eta_1} & \bar{\eta_2}
	\end{pmatrix}
	\begin{pmatrix}
	\rho_1 & \rho_2 \\
	\bar{\rho_1} & \bar{\rho_2}
	\end{pmatrix}^{-1}
	\begin{pmatrix}
	z \\ \bar{z}
	\end{pmatrix}
	\right],
	\end{equation}
	
	where $\rho_i, \eta_i, i=1,2$ denote, respectively, the periods and the quasi-periods of $E_{\lambda}$, and
	
	\begin{equation}\label{Beta}
	\begin{pmatrix}
	\d \beta_1 \\ \d \beta_2
	\end{pmatrix}=
	A^{-1}\left(\d(A)A^{-1}\begin{pmatrix}
	z \\ \bar{z}
	\end{pmatrix}+\begin{pmatrix}
	\d z \\ \d \bar{z}
	\end{pmatrix}\right),  \text{ where } \ A:= \begin{pmatrix}
	\rho_1 & \rho_2 \\
	\bar{\rho_1} & \bar{\rho_2}
	\end{pmatrix}.
	\end{equation}
	Let $\gamma_{\lambda}$ be the $1$-form $\d \lambda/2\lambda$. Using the relations $\frac{\d \rho_i}{\d \lambda}=\frac{1}{2\lambda} \eta_i$, we see that:
	\[
	\d (A)A^{-1}= \begin{pmatrix}
	\gamma_{\lambda} & 0 \\ 0 & \overline{\gamma_{\lambda}}
	\end{pmatrix}\begin{pmatrix}
	\eta_1 & \eta_2 \\
	\bar{\eta_1} & \bar{\eta_2}
	\end{pmatrix}
	\begin{pmatrix}
	\rho_1 & \rho_2 \\
	\bar{\rho_1} & \bar{\rho_2}
	\end{pmatrix}^{-1}.
	\] 
	To simplify the notation, we define:
	\[ 
	C:= \begin{pmatrix}
	\ C_1 \ \ \\  \ C_2 \ \
	\end{pmatrix}:= \begin{pmatrix}
	C_{11} & C_{12} \\ C_{21} & C_{22}
	\end{pmatrix}
	:= \begin{pmatrix}
	\eta_1 & \eta_2 \\
	\bar{\eta_1} & \bar{\eta_2}
	\end{pmatrix}
	\begin{pmatrix}
	\rho_1 & \rho_2 \\
	\bar{\rho_1} & \bar{\rho_2}
	\end{pmatrix}^{-1}
	\]
	\[
	=\frac{1}{\det (A)}
	\begin{pmatrix}
	\eta_1 \bar{\rho_2}-\eta_2\bar{\rho_1} & 2\pi i \\ 2 \pi i& -\bar{\eta_1}\rho_2+\bar{\eta_2}\rho_1
	\end{pmatrix}.\footnote{We have used here the Legendre relation $\rho_2\eta_1-\rho_1\eta_2= \pm 2 \pi i$}
	\]
	We wish to express now both $\d \beta_1 \wedge \d \beta_2$ and $T$ as linear combinations (with coefficients in $\mathcal{C}^{(\infty)}(\pi^{-1}(U))$, $U \subset \P_1 \setminus \{0,1,\infty\}$ being an open simply connected domain) of the forms $\gamma_{\lambda} \wedge \overline{\gamma_{\lambda}}, \gamma_{\lambda} \wedge \d \bar{z},\d z \wedge \overline{\gamma_{\lambda}},\d z \wedge \d \bar{z}$. We note that a (smooth) $(1,1)$-form in $\Omega^{1,1}(\pi^{-1}(U))$ can be written uniquely as a linear combination of these forms, since the two $1$-forms $\gamma_{\lambda}$ and $\d z$ span the cotangent bundle over each point in $\pi^{-1}(U)$. Hence it will suffice to check that the two expressions for $\d \beta_1 \wedge \d \beta_2$ and $T$ are the same.
	
	We first do this for $\d \beta_1 \wedge \d \beta_2$.
	
	Namely, using (\ref{Beta}), one deduces that:
	\begin{equation}\label{Betaform}
	\d \beta_1 \wedge \d \beta_2=\frac{1}{\det (A)}\left[(C_{11}z+C_{12}\bar{z}){\gamma_{\lambda}}+\d z\right]\wedge\left[(C_{21}z+C_{22}\bar{z}){\overline{\gamma_{\lambda}}}+\d \bar{z}\right].
	\end{equation}
	
	We turn now to $T$.
		
	One can verify that:
	\[
	\partial C_2=\frac{-2\pi i{\gamma_{\lambda}}}{\det(A)}C_1, \quad \bar{\partial} C_1=\frac{-2\pi i\overline{\gamma_{\lambda}}}{\det(A)}C_2,
	\]
	\[
	\partial \bar{\partial} C_1=\frac{2\pi i}{\det (A)}[-C_{21}C_1-C_{11}C_2]\gamma_{\lambda}\wedge \overline{\gamma_{\lambda}}, 
	\]
	\[  \bar{\partial} \partial C_2=\frac{2\pi i}{\det (A)}[C_{12}C_2+C_{22}C_1]\gamma_{\lambda}\wedge \overline{\gamma_{\lambda}}.
	\]
	
	Using the Leibniz rule in (\ref{Expr:Traw}), and the above expressions, one gets:
	\[
	T=\frac{1}{4\pi i}\cdot \frac{1}{\det (A)}\left[ 4 \pi i\d z \wedge \d \bar{z}+2 \begin{pmatrix}
	\d z \\ \d \bar{z}
	\end{pmatrix}\wedge \begin{pmatrix}
	\  - 2 \pi iC_2 \overline{\gamma_{\lambda}} \ \ \\  \ { 2 \pi i }C_1 \gamma_{\lambda} \ \
	\end{pmatrix}\begin{pmatrix}
	 z \\  \bar{z}
	\end{pmatrix}\right.
	\]
	\begin{equation}\label{Expr:T}
		\left.+{2 \pi i}\begin{pmatrix}
		z \\ \bar{z}
		\end{pmatrix}\begin{pmatrix}
		\ C_{21}C_1+C_{11}C_2 \ \  \\ C_{12}C_2+C_{22}C_1
		\end{pmatrix}\begin{pmatrix}
		z \\ \bar{z}
		\end{pmatrix}\gamma_{\lambda}\wedge \overline{\gamma_{\lambda}}\right].
	\end{equation}
	
	An easy term-by-term comparison reveals that the expressions (\ref{Expr:T}) and (\ref{Betaform}) are the same.
\end{remark}

We focus now on giving a more conceptual proof of the equality $T=\d \beta_1 \wedge \d \beta_2$.

\begin{proposition}\label{Prop:uniqueform}
	Let $D \subset \P_1\setminus \{0,1,\infty\}$ be an open simply connected domain. The $(1,1)$-form $\d \beta_1 \wedge \d \beta_2$, is, up to scalar multiplication, the unique closed $2$-form in $\Omega_0^2(\pi^{-1}(D))$\footnote{We are using the notation $\Omega_0^2$ to denote continuous $2$-forms. } that satisfies $[2]^*\omega=4\omega$, where  $[2]:\pi^{-1}(D) \rightarrow \pi^{-1}(D)$ denotes the endomorphism of multiplication by $2$ on the fibers.
\end{proposition}
\begin{proof}
	Let $\omega \in \Omega_0^2(\pi^{-1}(D))$ be a $2$-form such that $[2]^*\omega=4 \omega$.
	
	Let:
	\[
	\omega=\sum_{i,j}\alpha_{ij} \gamma_i\wedge \gamma_j, 
	\quad \text{where } i,j \in \{1,2,\lambda,\bar{\lambda}\} \text{, } \ \begin{cases}
	\gamma_i=\d \beta_i \ \ i=1,2 \\
	\gamma_i=\d \lambda \  \ i=\lambda \\
	\gamma_i= \d \bar{\lambda} \ \ i=\bar{\lambda}
	\end{cases},
	\]
	and $\alpha_{ij} \in \mathcal{C}^0(\pi^{-1}(D))$.
	
	
	We notice that: 
	\[
	\frac{1}{4}[2]^*\omega=\sum_{ij}\alpha_{ij}([2]P) 2^{\delta{i}+\delta{j}-2}\gamma_i \wedge \gamma_j,
	\]
	where $\delta_i=1$ when $i=1,2$, and $\delta_i=0$ otherwise.
	We know that $\frac{1}{4}[2]^*\omega=\omega$, and, hence:
	\begin{equation}\label{Eq:equivariantforms}
	\alpha_{ij}(P)=2^{\delta{i}+\delta{j}-2}\alpha_{ij}([2]P), \quad \text{for all } i,j \in \{1,2,\lambda,\bar{\lambda}\} \text{ and } P \in \pi^{-1}(D').
	\end{equation}
	However, for each $x \in D$, restricting both hand sides of equation (\ref{Eq:equivariantforms}) to the fiber $\pi^{-1}(x)$, and then taking the $\infty$-norm yields:
	\[
	\max_{P \in \pi^{-1}(x)} |\alpha_{ij}(P)|=2^{\delta{i}+\delta{j}-2}\max_{P \in \pi^{-1}(x)} |\alpha_{ij}(P)|\leq 2^{-1}\max_{P \in \pi^{-1}(x)} |\alpha_{ij}(P)| \ \text{if } \{i,j\}\neq \{1,2\}.
	\]
	Hence, for $\{i,j\}\neq \{1,2\}$, $\alpha_{ij}\equiv 0$. Therefore, we have that:
	\[
	\omega=\alpha_{12}\d \beta_1 \wedge \d \beta_2,
	\]
	\[
	\alpha_{12}([2]P)=\alpha_{12}(P) \quad \forall P \in \pi^{-1}(D').
	\]
	
	Since the function $\alpha_{12}$ is continuous, this implies that $\alpha_{12}$ is constant on the fibers of $\pi$, and hence it depends just on $\lambda$ and $\bar{\lambda}$. Since $\omega$ is closed by hypothesis, it follows that $\alpha_{12}$ is constant in $\lambda$ and $\bar{\lambda}$ as well
	. Hence $\omega = c \ \d \beta_1 \wedge \d \beta_2$, with $c \in \C$.
\end{proof}

\begin{remark}\label{Rmk:hypothesis}
	In Proposition \ref{Prop:uniqueform}, one can replace the \textit{closed} hypothesis with the hypothesis that the restriction of $\omega$ to the fibers of $\pi$ is the normalized Haar measure. We note that the two different hypothesis do not \textit{a priori} imply each other. On the other hand, the proof that was presented here works in both cases. In fact, the only point in which we used the closed assumption was to deduce that the function $\alpha_{12}$ is constant. However, this is automatic if $\omega$ restricts to the normalized Haar measure on the fibers.
\end{remark}

\begin{remark}
	In Proposition \ref{Prop:uniqueform}, the hypothesis that $\omega$ is continuous on $\pi^{-1}(D)$ is crucial. The $(1,1)$-form $\d z \wedge \d \bar{z}$, which is not a constant multiple of $\d \beta_1 \wedge \d \beta_2$, satisfies all the hypothesis of the proposition, except that it is not a well-defined continuous $2$-form on $\pi^{-1}(D)$. This last fact may be easily seen by noticing that summing a period $\rho(\lambda)$ to $z$ changes the $2$-form $\d z \wedge \d \bar{z}$ by a (non-zero) term $\rho'(\lambda)\overline{\rho'(\lambda)}\d \lambda \wedge \d \bar{\lambda}+\overline{\rho'(\lambda)} \d z \wedge  \d \bar{\lambda}- \rho'(\lambda) \d \bar{z} \wedge \d \lambda $. 
\end{remark}

\begin{corollary}\label{Cor:dynamicalforT}
	The restriction to $E \setminus \mathcal{S}$ of the $(1,1)$-current $T=\frac{1}{2\pi i}\d \d^c {H_N}$ is equal to the current associated to the $(1,1)$-form $\d \beta_1 \wedge \d \beta_2$.
\end{corollary}
\begin{proof}
	Let $D \subset \P_1\setminus S$ be an open relatively compact simply connected domain. We notice that $T=\frac{1}{2\pi i}\d \d^c {H_N}$ is obviously closed. Moreover, we have that:
	\[
	([2]^*T)(P)=\frac{1}{2\pi i}\d \d^c (H_N([2]P)), \ \forall\  P \in \pi^{-1}(D).
	\]
	We have the following well-known equality (see e.g. \cite[Theorem VI.1.1]{silverman1994advanced}):
	\[
	H_N([2]P)=4H_N(P)-\log |2y|+\frac{1}{4}\log |\Delta_\lambda|,
	\]
	where a Weierstrass form for the elliptic curve $E_\lambda$ is assumed to have been fixed. Since we are working in the Legendre family, we may, of course, choose the Legendre form.
	
	Since $\d \d^c \log |f|=0$, for any holomorphic function $f$, this implies that:
	\[
	([2]^*T)(P)=\frac{1}{2\pi i}\d \d^c (H_N([2]P)))=4\frac{1}{2\pi i}\d \d^c H_N(P)=4T, \ \forall\  P \in \pi^{-1}(D).
	\]
	
	We prove now that ${T}|_{E\setminus \mathcal{S}}$ is smooth. We already know that $T$ is smooth away from the zero-section $O$. Let $\sigma_2:D \rightarrow E$ denote a section of $\pi$ of order $2$ (i.e. $[2]\sigma_2=O$, $\sigma_2 \neq O$). By restricting the equality $([2]^*T)(P)=4T(P)$ to a neighborhood of the section $\sigma_2$, we get a smooth $(1,1)$-current on the right hand side, hence the left hand side has to be smooth as well. Since the map $[2]:E\setminus \mathcal{S}  \rightarrow E\setminus \mathcal{S}$ defines a biholomorphism between a neighborhood of $\sigma_2$ and a neighborhood of $O$ (since we are restricting to the good reduction locus), this in turn tells us that the restriction of $T$ to a neighborhood of $O$ is smooth. Hence the restriction of $T$ to $E \setminus \mathcal{S}$ is smooth, and therefore it is represented by a $(1,1)$-form, which, with a slight abuse of notation, we will still denote by $T$.
	
	Hence, by Proposition \ref{Prop:uniqueform}\footnote{Here, instead of using Proposition \ref{Prop:uniqueform}, we could use its modified version, as in Remark \ref{Rmk:hypothesis}.}, $T=c \  \d \beta_1 \wedge \d \beta_2$, where $c \in \C$ is a constant. Since, as remarked at the beginning of this subsection, both $T$ and $\d \beta_1 \wedge \d \beta_2$ restrict to the normalized Haar measure on the fibers of $\pi$, $c=1$, as we wanted to prove.
\end{proof}

\subsection{An alternative proof of Theorem \ref{Prop:equality}}

As mentioned in the beginning of subsection \ref{Hic}, we sketch here another argument, using the work of DeMarco and Mavraki, that reproves Theorem \ref{Prop:equality}.

Let $D_{E}(\sigma):= \sum_{\gamma \in B(\bar{K})} \hat{\lambda}_{E, { ord }_{\gamma}}(\sigma) \cdot(\gamma)$, where \(\hat{\lambda}_{E, { ord }_{\gamma}}(\sigma)\) is the local canonical height of the point corresponding to \(\sigma\) on the elliptic curve \(E\) over
\(k=\C(B)\) at the place \( ord _{\gamma}\) (i.e. the place corresponding to the point $\gamma \in B$), for each \(\gamma \in B(\C) .\) The degree of \(D_{E}(\sigma)\) is equal to \(\hat{h}_{E}(\sigma) .\)  Then, DeMarco and Mavraki \cite[sec. 3,4]{DMM} prove that (keeping the notation above): \[\c_1(D_E(\sigma))=\sigma^*T=\sigma^*{\d \beta_1 \wedge \d \beta_2},\] where the last equality is Corollary \ref{Cor:dynamicalforT}, i.e. the comparison result proven in this section. Hence, Theorem \ref{Prop:equality} becomes a consequence of Wirtinger's formula applied to the divisor $D_E(\sigma)$. 



\begin{remark}\label{Rmk:Higher_dimension}
	In higher dimension (i.e. for a fibration $\mathcal{A}\rightarrow S$ in abelian varieties, where $\dim \mathcal{A} = 2\cdot \dim (S)=2g >1$) the situation is slightly different. For instance, one could not expect a formula as (\ref{soughtequality}) because $\hat{h}([n] \sigma)$ is always quadratic in $n \in \Z$, while $\int_{B \setminus r^{-1}(S)} ([n] \sigma)^*(\d \beta_1 \wedge \dots \wedge \d \beta_{2g})$ has degree $2g>2$ in $n \in \Z$.
\end{remark}

\begin{section}{Quasi-integral points}   \label{Sec4}

We start by interpreting Theorem \ref{Thm:Theorem2.6} in terms of points on the elliptic curve that are almost integral. But from a certain point of view the result is not optimal, and the main purpose of this section is to remedy this defect.
	
	\subsection{Introduction}
	
	Theorem \ref{Thm:Theorem2.6} may be reformulated in terms of heights on ${\C}(B)$. We may identify the section $\sigma$ with a point $(\xi,\eta)$ on $E({\C}(B))$. The arguments in section \ref{cmb} will make it clear that $m_\sigma(b)=-v(\xi)/2$, at least if $-v(\xi) \geq 0$ is large enough, where $v$ is the valuation corresponding to $b$. 
	
	For the Legendre model we will also see that
	\begin{equation}\label{456*a}
	h(\Xi(\sigma)) \leq 4h(\xi)+c_0
	\end{equation}
	for the natural height $h(\xi)=\sum_v\max\{0,-v(\xi)\}$ on ${\C}(B)$ and some $c_0$ (here absolute). So for this model we deduce from Theorem \ref{Thm:Theorem2.6} that $|\xi|_v=\exp(-v(\xi))$ satisfies
	\begin{equation}\label{456**a}
	|\xi|_v \leq e^{2c_0}H(\xi)^8
	\end{equation}
	with
	\begin{equation}\label{4562.5}
	H(\xi)=\exp(h(\xi))=\prod_v\max\{1,|\xi|_v\}.
	\end{equation}
	But of course (\ref{456**a}) is worse than the trivial
	\begin{equation}\label{456t}
	|\xi|_v \leq H(\xi).
	\end{equation}
	And we will also see that the 4 in (\ref{456*a}) cannot be avoided.
	
	Further we will see (for general models) that for sections $\sigma=n\sigma_0$ with $\sigma_0$ fixed, then thanks to $m_\sigma(b)=m_{\sigma_0}(b)$, Theorem \ref{Thm:Theorem2.6} gives bounds for $m_{\sigma}(b)$ independent of $n$; and in fact similar remarks hold for linear combinations like $n\sigma_0+m\tau_0$. Of course $E({\C}(B))$ is finitely generated so every $(\xi,\eta)$ corresponds to such a linear combination; but as finding a basis for the Mordell-Weil group remains an ineffective procedure, this does not provide an effective non-trivial improvement of (\ref{456**a}) in general.
	
	The main object of the present section is to provide such an improvement. We shall drop the geometric terminology and replace ${\C}(B)$ by a function field $\cal K$ of transcendence degree 1 over an algebraically closed field $k$ of zero characteristic.
	
	Thus let $E$ be an elliptic curve over $\cal K$, defined by say
	$$y^2=x^3+ax^2+bx+c$$
	(now without danger of confusing a coefficient with a point of some $B$). As in section 1, we assume (just for convenience) that the cubic factorizes completely over $\cal K$.
	
	Now an argument of Manin \cite{M} (see also Voloch \cite{JFV}) using formal groups shows that (\ref{456t}) can be replaced by
	\begin{equation}\label{4562}
	|\xi|_v \leq C_1
	\end{equation}
	where $C_1$ depends only on $E$ and $v$. However (\ref{4562}) is again not effective, and not just because of Mordell-Weil.
	
	In fact it is already implicit in the literature that an effective bound
	\begin{equation}\label{456tt}
	|\xi|_v \leq C_2H(\xi)^\theta
	\end{equation}
	holds for some absolute $\theta<1$ with $C_2$ depending only on $E$ and $\cal K$. Take for simplicity ${\cal K}$ as $\C(t)$ and suppose we have absolute constants $\gamma, \delta$ such that
	\begin{equation}\label{INT}
	\deg X \leq \gamma\max\{\deg A,\deg B, \deg C\}+\delta
	\end{equation}
	for all ``integral'' $X,Y$ in $\C[t]$ with $Y^2=X^3+AX^2+BX+C$. Then writing $\xi=X/Z^2,\eta=Y/Z^3$ for our point $(\xi,\eta)$ of $E(\C(t))$, we deduce easily (\ref{456tt}) for $\theta=1-{1 \over 3\gamma}$.
	
	Now (\ref{INT}) has been known for some time, and for example Theorem 6 of Mason \cite{RCM} (p.30) leads to $\theta={155\over156}$ (for general $\cal K$). But this procedure is wasteful, and a more detailed direct analysis yields $\theta={7 \over 8}$. Here the arguments use the $abc$ inequality for function fields, whereas our arguments use a refined version (due to Wang) involving carefully chosen $abcd\ldots$.
	
	Here is our improvement of (\ref{456t}) and (\ref{456tt}). To highlight the effectivity we give a completely explicit bound depending for example on a height
	$$H(E)=\prod_v\max\{1,|a|_v^6,|b|_v^3,|c|_v^2\}.$$ 
	Crucial for certain applications will be the fact that it no longer depends on the valuation $v$ in (\ref{456t}). From now on all valuations will be supposed to have value group $\bf Z$ when written additively.
	\begin{theorem}\label{qi}
		Given an elliptic curve $E$ over a function field $\cal K$ as above, and any $\e>0$, there is an effective constant $C=C(E,\cal K,\e)$, depending only on $E,\cal K,\e$, such that
		$$|\xi|_v \leq CH(\xi)^\e$$
		for any $(\xi,\eta)$ on $E(\cal K)$ and any valuation $v$ on $\cal K$. In fact if $\e \leq 1/16$ then we may take $C=\left(e^gH(E)\right)^\rho$
		where $g$ is the genus of $\cal K$, and $\rho=2^{10000/\e^2}$.
	\end{theorem}
	
	Note that the same bound is of course valid for $\max\{1,|\xi|_v\}$. Our methods will almost certainly establish
	$$\prod_{v \in S}\max\{1,|\xi|_v\} \leq (e^{g-1+|S|}H(E))^\rho H(\xi)^\e$$
	for any finite set $S$ of valuations of $\cal K$. This would not follow simply by multiplying the individual bounds in the theorem.
	\bigskip

	Our proof essentially follows the classical strategy of Siegel. If $\xi$ is $v$-adically large then $P=(\xi,\eta)$ is close to the origin $O$. Regarding $\e$ as fixed, we fix a positive integer $m$ (later to be related to $\e$) and note that for every $w$ over $v$ of a suitable field extension some submultiple $P/m$ is $w$-adically close to some torsion point $T=O/m$. Here the point is that dividing by $m$ does not essentially worsen the closeness. If $P/m=(\xi_m,\eta_m)$ and $T=(\tau,\omega)$ then the fixed quantity $\tau$ algebraic over $\cal K$ is well-approximated by the varying quantity $\xi_m$ also algebraic over $\cal K$, and we are set up for Roth's Theorem (which of course Siegel did not have), taking all the available $w$. As in Siegel's strategy, we win because the logarithmic height $h(\xi_m)$ gets a lot smaller, about ${1 \over m^2}h(\xi)$. This means that we do not need the arbitrary Roth exponent $\kappa>2$, and for example $\kappa=3$ would suffice. And it is slightly simpler technically first to add a fixed point of order 2 to $P$ and work with a finite target; this rules out the possibility $T=O$.
	
	Actually Siegel argued slightly differently, using the Mordell-Weil Theorem that $E(\cal K)$ is finitely generated to write $P=mQ+R$ with $Q$ in $E(\cal K)$ and a remainder $R$; then the fixed algebraic $-R/m$ is approximated by $Q$ which needs no field extension. We can avoid this because the function field versions of Roth are much more uniform, so we do not suffer from the Mordell-Weil non-effectivity.
	
	That Roth's Theorem for function fields is effective seems to have been proved first by Osgood \cite{CFO2}, at least for a single valuation; and Wang \cite{W} treated several valuations, on the way considerably simplifying the proof. In subsection \ref{wr} we shall give some extra minor simplifications in the proof of her Main Theorem and its applications to Roth's Theorem.
	
	Then in subsection \ref{ex} we highlight the effectivity by giving an example in more traditional form for a particular algebraic function of degree 4 over ${\C}(t)$. We even calculate all the implied constants effectively.
	
	In subsection \ref{pre} we record some observations preliminary to the proof of our Theorem \ref{qi}, which then follows in subsection \ref{thmqi}.

	\subsection{Wang's version of Roth's Theorem}\label{wr}
	
	We stay with a function field $\cal K$ of transcendence degree 1 over an algebraically closed field $k$ of zero characteristic. It has a genus $g \geq 0$. We normalize the valuations $v$ on $\cal K$ such that $v(f)=-\log|f|_v$ has value group $\bf Z$. Then the logarithmic height
	$$h(f)=-\sum_v\min\{0,v(f)\}, $$
	corresponding to (\ref{4562.5}), is also given by
	\begin{equation}\label{456m}
	h(f)=[{\cal K} :k(f)]. 
	\end{equation}
	
	We fix a finite set $S$ of these valuations and define
	$$\chi=2g-2+|S|$$
	for the cardinality of $S$.
	
	Now we take a finite set $A^*$ consisting of $0$ together with a non-empty set of $S$-units. 
	
	For an integer $r \geq 1$ denote by $L(r)$ the vector space over $k$ spanned by monomials of degree $r$ in the elements of $A^*$, and write $l(r) \geq 1$ for its dimension. Define also $L(0)=k$, so $l(0)=1$. Note that 
	$$a^*L(r) \subseteq L(r+1)$$ 
	for every $a^*$ in $A^*$. Here is a version of the Main Theorem of \cite{W} (p.1226).
	
	\begin{lemma}\label{l1}  
		For each $v$ in $S$ choose some $a_v^*$ in $A^*$. Suppose $f \neq 0$ is in $\cal K$ and $r \geq 0$ are such that
		\begin{equation}\label{456hyp}
		fL(r) \cap L(r+1)=\{0\}.
		\end{equation}
		Then $f$ is not in $A^*$ and we have
		$$\sum_{v\in S}\max\{0,v(f-a_v^*)\}+\sum_{v \notin S}\max_{a^* \in A^*}\max\{0,v(f-a^*)-(m+n-1)\}\le {m+n\over n}h(f)+{(m+n)(m+n-1)\over 2n}\chi$$
		for $n=l(r)$ and $m=l(r+1)$.
	\end{lemma}
	In fact the (non-negative) sum over $v \notin S$ does not appear in \cite{W}; here it allows our result to be considered as an analogue of Nevanlinna's Second Main Theorem with ramification. We have also eliminated some extra heights appearing in \cite{W}.
	\begin{proof}
		If $f$ were in $A^*$, then $fL(r)$ would lie in $L(r+1)$, forcing $fL(r)=\{0\}$ by $(\ref{456hyp})$, a contradiction.
		
		Choose basis elements $\beta_1,\ldots ,\beta_n$ of $L(r)$ and basis elements $b_1,\ldots ,b_m$ of $L(r+1)$. Write
		$$\mu_v=\max_{a^* \in A^*}\max\{0,v(f-a^*)-(m+n-1)\}$$
		as in the lemma. As in \cite{W}, we let $t$ be a nonconstant element of $\cal K$, we let $t_v$ be a local parameter at $v$, and we consider $\omega,\omega_v$ defined as the Wronskians with respect to $t,t_v$ respectively of $f\beta_1,\ldots ,f\beta_n,b_1,\ldots ,b_m$. Our assumption (\ref{456hyp}) together with $f \neq 0$ and $n \geq 1$ imply that these latter are linearly independent over $k$. So $\omega\neq 0,\omega_v \neq 0$.
		
		We let $S_1$ be the subset of $S$ made up of those $v$ such that $v(f-a_v^*)>0$; clearly the sum on the far left can be restricted to $S_1$.
		
		We estimate $v(\omega_v)$ according to five disjoint cases for $v$.
		
		Case (i): $v\not \in S, v(f)\ge 0$.  
		
		If $\mu_v=0$ we use just 
		\begin{equation}\label{456a1}
		v(\omega_v)\ge 0.
		\end{equation}
		
		If $\mu_v>0$ then there is $a^*$ in $A^*$ with $v(f-a^*)=m+n-1+\mu_v$. We note as in \cite{W} that $\omega_v$ is also the Wronskian of $(f-a^*)\beta_1,\ldots,(f-a^*)\beta_n,b_1,\ldots,b_m$ (with respect to $t_v$). This is because $a^*\beta_1,\ldots,a^*\beta_n$ are in $L(r+1)$ and so we may use column operations. 
		
		Now by looking at the first $n$ columns we see that
		$$v(\omega_v) \geq \mu_v+(\mu_v+1)+\cdots+(\mu_v+n-1) \geq n\mu_v;$$
		thus from (\ref{456a1}) we have
		$$v(\omega_v)\ge n\mu_v.$$
		for all $v$ in this case (i).
		
		Case (ii):   $v\not \in S$, $v(f)<0$. 
		
		Now standard identities show that $\omega_v/f^{m+n}$ is the Wronskian of $\beta_1,\ldots ,\beta_n,b_1/f,\ldots ,b_m/f$ (with respect to $t_v$).  These functions are all regular at $v$, hence 
		$$v(\omega_v)\ge (m+n)v(f).$$
		
		Case (iii): $v\in S\setminus S_1$, $v(f)<0$.  
		
		We use the same formula as in Case (ii). Now taking into account possible poles of the functions $\beta_1,\ldots ,\beta_n,b_1,\ldots ,b_m$, the  usual computation yields
		$$v(\omega_v)\ge (m+n)v(f)+\sum_{i=1}^nv(\beta_i)+\sum_{j=1}^mv(b_j)-B.$$
		for the binomial coefficient $B={m+n\choose 2}$.
		
		Case (iv): $v\in S\setminus S_1$, $v(f)\ge 0$.  
		
		The usual computation on the original Wronskian yields
		$$v(\omega_v)\ge\sum_{i=1}^nv(\beta_i)+\sum_{j=1}^mv(b_j)-B.$$
		
		Case (v): $v\in S_1$, that is, $v(f-a_v^*)>0$. 
		
		As above in (i) we can replace $f$ in the Wronskian by $f-a_v^*$. Then  as in (iv), but now taking into account $f-a_v^*$ multiplying $\beta_1,\ldots ,\beta_n$, we get 
		$$v(\omega_v)\ge nv(f-a_v^*)+\sum_{i=1}^nv(\beta_i)+\sum_{j=1}^mv(b_j)-B.$$
		
		This completes the analysis, as we have covered all possible  cases.
		
		Let us now sum over all $v$ on $\cal K$, and use the five inequalities so obtained, noting that the respective sets give a partition of all $v$. We obtain as lower bound for $\sum_vv(\omega_v)$
		$$n\sum_{v \notin S,v(f) \geq 0}\mu_v+n\sum_{v\in S_1}v(f-a_v^*)+ (m+n)  \sum_{v\not \in S_1,v(f)<0}v(f) +\sum_{i=1}^n\sum_{v\in S}v(\beta_i)+\sum_{j=1}^m\sum_{v\in S}v(b_j)-B|S|.$$
		Finally, in the sum over $\mu_v$ we may omit $v(f)\geq 0$, since $\mu_v=0$ when $v(f) < 0$. Also the sum $\sum_{v\not \in S_1,v(f)<0}v(f)$ is at least $\sum_{v(f)<0}v(f)$, which in turn is the number of poles of $f$ counted with (negative) multiplicity, hence equals $-h(f)$; the two subsequent   double sums vanish because $\beta_i,b_j$ are $S$-units. Also, by a well-known formula for Wronskians, we have $\omega_v=\omega({d} t/{d}t_v)^B$. Therefore, since $\sum_vv(\omega)=0$, we see that $\sum_vv(\omega_v)=B(2g-2)$ (by the Hurwitz formula). 
		
		This immediately leads to the stated inequality and completes the proof.
	\end{proof}
	
	We turn now to the application to Roth's Theorem for function fields, which as in \cite{W} involves the elimination of (\ref{456hyp}). But here we also drop all references to $S$-units. 
	
	Thus we take a finite set $A$ consisting of $0$ together with a non-empty set of $l$ non-zero elements of $\cal K$. 
	\begin{proposition}\label{roth} For each $v$ in $S$ choose some $a_v$ in $A$. Then for $f \neq 0$ in $\cal K$ and any positive $\e \le 1/16$ we have either
		$$h(f)\le  {6l\over \e}\left(\log {1\over \e}\right)\sum_{a\in A}h(a)$$
		or $f$ is not in $A$ and
		$$\sum_{v\in S}\max\{0,v(f-a_v)\}\le (2+\e)h(f)+3\left({1\over \e}\right)^{l}(\chi+2\sum_{a\in A}h(a)).$$
	\end{proposition}
	\begin{proof} We start by enlarging the set $S$ to make the non-zero elements $a$ of $A$ into $S$-units. So for each such $a$ we must throw in the $v$ with $v(a) \neq 0$. Their number is at most
		\begin{equation}\label{456*}
		\sum_{v(a)>0}v(a)-\sum_{v(a)<0}v(a) \leq 2h(a). 
		\end{equation}
		Now we are set up to apply Lemma \ref{l1}.
		
		We have to find an $r$ with $l(r+1)/l(r)\le 1+\e$. If this fails for say $r=0,1,\ldots ,R-1$, then we get
		$$
		(1+\e)^R<l(R) \leq {R+l-1\choose l-1}.
		$$
		Write $\delta=\sqrt{1+\e}-1$. We have 
		\begin{equation}\label{456f}
		{R+l-1\choose l-1}\le (1+\delta)^{R+l-1}\delta^{1-l},
		\end{equation}
		hence
		$(1+\e)^{R/2}<(1+\e)^{l-1\over 2}\delta^{1-l}$, which yields 
		$$R\le l-1+2(l-1){\log (1/\delta)\over \log (1+\e)}\leq R'$$
		for $R'=3(l-1)(1 / \e)\log {1 / \e}$. 
		
		So if we  choose 
		$$R=[R']+1\leq {3l\over \e}\log {1\over \e},$$
		we can find $r\le R-1$ as above, that is, with $l(r+1)\le (1+\e)l(r)$. We then apply Lemma \ref{l1} with $m=l(r+1)$, $n=l(r)$. If (\ref{456hyp}) is not satisfied, then with the bases $\beta_1,\ldots ,\beta_n,b_1,\ldots ,b_m$ as before, we see that $f\beta_1,\ldots ,f\beta_n,b_1,\ldots ,b_m$ must be linearly dependent over $k$. So by Lemma 5 of \cite{W} (p.1232), we have
		$$h(f)\le (2r+1)\sum_{a\in A}h(a)\le 2R\sum_{a\in A}h(a)\le {6l\over \e}\left(\log {1\over \e}\right)\sum_{a\in A}h(a),$$
		which is the first of the two alternative conclusions of the present proposition.
		
		Therefore we may indeed assume that the conclusion of Lemma \ref{l1} holds. 
		
		Now by (\ref{456f})
		$$\log{R+l-1\choose l-1}\le (R+l-1)\log(1+\delta)+(l-1)\log{1 \over \delta}$$
		which is at most
		$$l\left({4 \over \e}\left(\log{1 \over \e}\right)\log(1+\delta)+\log{1 \over \delta}\right)<l\log{1 \over \e}.$$
		As 
		$$n = l(r) \leq {r+l-1\choose l-1}\leq {R+l-1\choose l-1}$$
		we deduce $n \leq (1/\e)^l$. Also 
		$${m+n \over n} \leq 2+\e \leq {33 \over 16}$$
		so
		$${(m+n)(m+n-1) \over 2n} \leq {33 \over 16}{m+n-1 \over 2} < 3n.$$
		We had enlarged the size of $S$ by at most $2\sum_{a \in A}h(a)$ from (\ref{456*}), and this completes the proof of Proposition \ref{roth}, as the left-hand side of the second alternative conclusion only gets bigger.
	\end{proof}

	\subsection{An example}\label{ex} 
	It is clear that by optimizing $\e$ in Proposition \ref{roth} we obtain something of the shape
	$$2h(f)+O(h(f)^{l/l+1})$$
	and therefore we obtain (effective) versions of Roth's Theorem that are stronger than the analogues over finite extensions of $\bf Q$. Similar results, even with $O(h(f)^{\theta})$ for any $\theta>2/3$, were found by Osgood \cite{CFO2} - see Theorem VIII (p.382) with $M=2$. Since \cite{CFO2} is written from a rather more general point of view, and also the constants are not always calculated, we feel it may be of interest to work out a completely explicit example. 
	
	So now we take $\alpha$ in a finite extension of say ${\C}(t)$, and for simplicity we want to bound the traditional $|\alpha-p/q|$ from below, where $p$ and $q \neq 0$ are in ${\C}[t]$ and the valuation extends that on ${\C}(t)$ defined by $v(1/t)=1$.
	
	If $\alpha$ has degree 2 over ${\C}(t)$, then of course we can avoid the $\e$ altogether. This is true also if $\alpha$ has degree 3 over ${\C}(t)$; it was noted first also by Osgood \cite{CFO}, even in effective form, and Schmidt \cite{WMS} worked these out in detail. An example is for
	$$\alpha=-{1 \over t}-{1 \over t^3}-{3 \over t^5}-{12 \over t^7}-{55 \over t^9}-\cdots$$
	satisfying $\alpha^3-\alpha={1 / t}$; then there is an obvious extension of $v$ to ${\C}(t,\alpha)$, and Theorem 1(i) of \cite{WMS} (p.2) implies that
	$$\left|\alpha - {p \over q}\right| \geq {e^{-6} \over |q|^2}.$$
	Actually we know of no obstacle to the conjecture that this can be done for $\alpha$ of any degree $d \geq 2$ over ${\C}(t)$. But already for $d=4$ the methods of \cite{WMS}, based on the use of differential equations, yield only $|q|^3$.
	
	We will work out a coresponding result for the example
	\begin{equation}\label{456alf}
	\alpha=-{1 \over t}+{1 \over t^4}-{4 \over t^7}+{22 \over t^{10}}-{140 \over t^{13}}+\cdots
	\end{equation}
	with 
	$$\alpha^4-\alpha={1 \over t}.$$

	{\bf Example.}  {\it We have
		$$\left|\alpha - {p \over q}\right| \geq {e^{-10^9} \over |q|^2}\exp\{-3(\log|q|)^{4/5}\}.$$
		for all $p$ and $q \neq 0$ in ${\C}[t]$.}

	{\it Verification.} We take $\cal K$ as the Galois closure of ${\C}(t,\alpha)$, with $[{\cal K}:{\C}(t)]=24$. We check that the genus $g=4$. In fact the above $v$ extends easily to $\cal K$, the other conjugates of $\alpha$ being
	\begin{equation}\label{456bet}\alpha'=1+{1/3 \over t}-{2/9 \over t^2}+\cdots,~~\alpha''=\omega+{1/3 \over t}+\cdots.~~\alpha'''=\omega^2+{1/3 \over t}+\cdots
	\end{equation}
	(for $\omega=e^{2\pi i/3}$) still in ${\C}((1/t))$.
	We get 24 valuations $v_\sigma$ on $\cal K$ given by $v_\sigma(x)=v(x^{\sigma})$ for $\sigma$ in the Galois group $S_4$ of ${\cal K}/{\C}(t)$. These will make up our set $S$. So $\chi=30$. Now with $f=p/q$
	$$24v(f-\alpha)=24v_\sigma(f-\alpha^{\sigma^{-1}})=\sum_{\sigma \in S_4} v_\sigma(f-\alpha^{\sigma^{-1}})$$
	and so our set $A=\{0,\alpha,\alpha',\alpha'',\alpha'''\}$ with $l=4$. By (\ref{456m}) we have
	$$h(\alpha)=[{\cal K}:{\C}(\alpha)]=[{\cal K}:{\C}(t,\alpha)]=6$$
	and the same for the other conjugates. So we deduce from Proposition \ref{roth} that either
	\begin{equation}\label{456u}
	h \leq {576 \over \e}\log{1 \over \e}
	\end{equation}
	or
	\begin{equation}\label{456v}
	24v(f-\alpha) \leq 2h+\e h+234\e^{-4}
	\end{equation}
	for $h=h(f)$ and any positive $\e \leq 1/16$.
	
	Optimizing $\e$ in (\ref{456v}) gives $\e=(936/h)^{1/5}$ so we need
	\begin{equation}\label{456w}
	h \geq 936.16^5. 
	\end{equation}
	If we temporarily assume this then (\ref{456u}) is impossible, and we conclude
	\begin{equation}\label{456x}
	24v(f-\alpha) \leq 2h+{5 \over 4}936^{1/5}h^{4/5}.
	\end{equation}
	Also $h(t)=[{\cal K}:{\C}(t)]=24$ so
	$$h=h\left({p \over q}\right)=24\max\{\deg p,\deg q\}$$
	The further assumption $v(f-\alpha)>0$ is harmless and then by (\ref{456alf}) we see that $\deg p < \deg q$ and so $h=24\deg q$. Now dividing (\ref{456x}) by 24 and then exponentiating gives
	$$\left|\alpha - {p \over q}\right| \geq {1 \over |q|^2}\exp\left\{-{5 \over 4}\left({936 \over 24}\right)^{1/5}(\log|q|)^{4/5}\right\}\geq {1 \over |q|^2}\exp\{-3(\log|q|)^{4/5}\}.$$
	
	When (\ref{456w}) fails we fall back on Liouville's argument; this amounts to
	$$|(\alpha-f)(\alpha'-f)(\alpha''-f)(\alpha'''-f)|=\left|f^4-f-{1 \over t}\right|=\left|{tp^4-tpq^3-q^4 \over tq^4}\right| \geq {e^{-1} \over |q|^4}.$$
	Again the harmless $|f-\alpha|<1$ implies $|\alpha'-f|=|\alpha''-f|=|\alpha'''-f|=1$ by (\ref{456bet}), so we deduce
	$$\left|\alpha - {p \over q}\right| \geq {e^{-1} \over |q|^4}=e^{-1}e^{-h/6} \geq \exp(-10^9),$$
	accounting for the extra factor in the Example. This completes the verification.

	Here the exponential factor beats that in the analogous Cugiani-Mahler-Bombieri-van der Poorten result \cite{BP} for algebraic numbers, which is 
	$$\exp\left\{-c\log|q|\left({{\log\log\log |q| \over \log\log |q|}}\right)^{1/4}\right\}$$ 
	(for ineffective $c$) - and furthermore there is a condition of slowly growing denominators.
	
	Here we may go beyond \cite{CFO2} by treating several valuations. This leads to new results of the type associated with Ridout. For example, with $v'(t)=1$ and $q$ a power of $t$ it is easy to show that immediately after a term $c/t^d$ in (\ref{456alf}) the number of consecutive zero coefficients is $O(d^{4/5})$.
	
	\subsection{Preliminaries}\label{pre}
	With the proof of our Theorem \ref{qi} in mind, we return now to our elliptic curve $E$ over $\cal K$ with equation $y^2=x^3+ax^2+bx+c$ for $a,b,c$ in $\cal K$ and logarithmic height 
	$$h(E)=\log H(E)=\sum_v\max\{0,-6v(a),-3v(b),-2v(c)\}.$$
	
	We will often use Zimmer's inequality
	\begin{equation}\label{456z}
	\left|h(P)-{3 \over 2}\hat h(P)\right| \leq {1 \over 2}h(E)
	\end{equation}
	for $P=(\xi,\eta)$ in $E({\cal K})$ with
	$$h(P)=\sum_v\max\{0,-v(\xi),-v(\eta)\},$$
	and $\hat h$ as before is with respect to twice the origin $O$. See \cite{HGZ} Proposition 11.1 (p.484).
	
	The $h$ is with respect to $3O$, and we will also need the easy (one-sided)
	\begin{equation}\label{comp}
	h(P) \leq {3 \over 2}h(\xi)+{1 \over 4}h(E).
	\end{equation}
	On the other side $h(P) \geq (3/2)h(\xi)-(1/2)h(E)$ could almost as easily be checked using the identity
	$$(x^2-ax+a^2-b)(x^3+ax^2+bx+c)-(a^3-2ab+c)x^2-(a^2b-ac-b^2)x-(a^2c-bc)=x^5$$
	but this we will not need.
	
	In our Theorem \ref{qi} we are implicitly considering $|\xi|_v$ to be large. Then $P=(\xi,\eta)$ is near $O$, so $P+Q_0$ is near $Q_0$ for any fixed $Q_0$. Thus the abscissae of $P+Q_0,Q_0$ are close. This is expressed precisely in the following result, where we choose $Q_0$ as a point of order 2 (recall these are defined over $\cal K$), and revert to the additive notation.
	\begin{lemma}\label{l2}For $Q_0=(\mu_0,0)$ and any $O\neq P=(\xi,\eta)\neq Q_0$ in $E(\cal K)$ we have for $P+Q_0=(\mu,\nu)$ the inequality
		$$\max\{0,v(\mu-\mu_0)\}\ge -v(\xi)-2h(E).$$
	\end{lemma}
	\begin{proof} Factor $x^3+ax^2+bx+c=(x-\mu_0)B(x)$. We have $\xi \neq \mu_0$ and we check $\mu -\mu_0={B(\mu_0) / (\xi-\mu_0)}$. Hence
		$$v(\mu-\mu_0)=v(B(\mu_0))-v(\xi-\mu_0).$$
		Now
		$$h(\mu_0) \leq h(Q_0) = h(Q_0)-{3 \over 2}\hat h(Q_0) \leq {1 \over 2}h(E)$$ 
		by (\ref{456z}). Hence 
		$$v(B(\mu_0))\ge -h(B(\mu_0))=-h(3\mu_0^2+2a\mu_0+b)\ge -2h(\mu_0)-{1 \over 3}h(E)\ge -2h(E).$$
		
		Also, $v(\xi-\mu_0)= v(\xi)$ if $v(\xi)<v(\mu_0)$. When this holds, we then have $v(\mu-\mu_0)\ge -v(\xi)-2h(E)$, proving the stated inequality. When this does not hold, then $-v(\xi)\le -v(\mu_0)\le h(\mu_0)\le h(E)/2$. 
		Hence $-v(\xi)-2h(E)\le 0$, thus completing the proof of the lemma. 
	\end{proof}
	
	We are now implicitly considering $Q=P+Q_0$ to be near a point of order 2. When we divide by a positive integer $m$, there are several possibilities for the quotients $Q/m$, and each of them should be almost as close (with a loss essentially independent of $m$) to a point of $E[2m]$, that is, a point of order dividing $2m$. This is expressed in the next key result, again in terms of abscissae. 
	\begin{lemma}\label{l3}With $Q_0=(\mu_0,0)$ suppose that for some odd integer $m \geq 3$ the points in $E[2m]$ are in $E({\cal K})$. Then for any $O\neq R=(\zeta,\theta)$ in $E(\cal K)$ with $O\neq Q=mR=(\mu,\nu)\neq Q_0$, we can find $O\neq T=(\tau,\omega)$ in $E[2m]$ with
		$$\max\{0,v(\zeta-\tau)\}\ge {1\over 2}\max\{0,v(\mu-\mu_0)\}- m^2h(E).$$
	\end{lemma}
	\begin{proof} Let $\varphi$ be the rational function representing multiplication by $m$ on abscissae; the numerator has degree $n=m^2$ and is monic, while the denominator has degree $n-1$ with leading coefficient $n$. We may therefore write
		$$\varphi(x)-\mu_0={1 \over n}{(x-\tau_1)\cdots (x-\tau_n)\over (x-\sigma_1)\cdots (x-\sigma_{n-1})}.$$
		Here $\tau_1,\ldots,\tau_n$ correspond to $\varphi(x)=\mu_0$, that is $mT=\pm Q_0=Q_0$; thus they are the abscissae of all such $T$, each in $E[2m]$. And $\sigma_1,\ldots,\sigma_{n-1}$ correspond to $\varphi(x)=\infty$, so are the abscissae of all non-zero points of $E[m]$.
		
		Thus putting $x=\zeta$ gives
		$$\mu-\mu_0={ 1\over n}{(\zeta-\tau_1)\cdots (\zeta-\tau_n)\over (\zeta-\sigma_1)\cdots (\zeta-\sigma_{n-1})}.$$
		
		We may assume $\tau=\tau_n$ satisfies
		\begin{equation}\label{456x1}
		v(\zeta-\tau)\ge v(\zeta-\tau_j)~~~(j=1,\ldots,n).
		\end{equation}
		We note that $h(\sigma_i) \leq h(E)/2$ and similarly
		\begin{equation}\label{456x2}
		h(\tau-\sigma_i) \leq h(\tau)+h(\sigma_i)\leq h(E)~~~(i=1,\ldots,n-1)
		\end{equation}
		\begin{equation}\label{456x3}
		h(\tau-\tau_j) \leq h(E)~~~(j=1,\ldots,n).
		\end{equation}
		
		We have
		\begin{equation}\label{456a5}
		v(\mu-\mu_0)=\sum_{j=1}^nv(\zeta-\tau_j)-\sum_{i=1}^{n-1}v(\zeta-\sigma_i)
		\end{equation}
		and so by (\ref{456x1}) and the ultrametric inequality
		\begin{equation}\label{456a5+}
		v(\mu-\mu_0)\leq nv(\zeta-\tau)-\sum_{i=1}^{n-1}\min\{v(\zeta-\tau),v(\tau-\sigma_i)\} 
		\end{equation}
		We will see that the factor $n$ can be reduced to 2 (giving the desired independence of $m$).
		
		We treat three possibilities for $v(\zeta-\tau)$.
		
		First suppose
		\begin{equation}\label{456x4}
		v(\zeta-\tau) < -h(E). 
		\end{equation}
		Now 
		\begin{equation}\label{456a6}v(\tau-\sigma_i) \geq -h(\tau-\sigma_i) \geq -h(E)
		\end{equation}
		by (\ref{456x2}), so by (\ref{456x4}) the minima in (\ref{456a5+}) are all $v(\zeta-\tau)$. So we obtain
		\begin{equation}\label{456b1}
		v(\mu-\mu_0) \leq v(\zeta-\tau)
		\end{equation}
		rather stronger than needed.
		
		Second suppose
		\begin{equation}\label{456a7}
		-h(E) \leq v(\zeta-\tau) \leq h(E). 
		\end{equation}
		Using the left-hand inequality in (\ref{456a7}) together with (\ref{456a6}) we find that all the minima in (\ref{456a5+}) are at least $-h(E)$, so the right-hand inequality in (\ref{456a7}) gives
		$$v(\mu-\mu_0) \leq nh(E)+(n-1)h(E) = 2nh(E)-h(E).$$
		Thus again the left part of (\ref{456a7}) gives
		\begin{equation}\label{456b2}
		v(\mu-\mu_0) \leq v(\zeta-\tau)+2nh(E)
		\end{equation}
		a little worse than (\ref{456b1}) but still better than needed.
		
		Our last possibility is
		\begin{equation}\label{456a8}
		v(\zeta-\tau) > h(E). 
		\end{equation}
		Now in (\ref{456a5}) we have
		\begin{equation}\label{456a9}
		v(\zeta-\tau_j) \geq \min\{v(\tau-\tau_j),v(\tau-\zeta)\}.
		\end{equation}
		Here $v(\tau-\tau_j) \leq h(\tau-\tau_j) \leq h(E)$ by (\ref{456x3}), so by (\ref{456a8}) the two values in the minimum in (\ref{456a9}) are distinct. Thus we actually have equality in (\ref{456a9}), and
		\begin{equation}\label{456a9+}v(\zeta-\tau_j) =v(\tau-\tau_j) \leq h(E).
		\end{equation}
		Now $\tau_1,\ldots,\tau_n$ are not all different; in fact the $T$ with $mT=Q_0$ come in pairs $\{T,-T\}$ together with $T=Q_0$. Thus our $\tau-\tau_n$ occurs $e$ times, where $e=1$ or $e=2$.
		
		Using (\ref{456a9+}) in (\ref{456a5}),(\ref{456a5+}) for the $\tau_j \neq \tau$, we get
		\begin{equation}\label{456a10}
		v(\mu-\mu_0)\leq ev(\zeta-\tau)+(n-e)h(E)-\sum_{i=1}^{n-1}\min\{v(\zeta-\tau),v(\tau-\sigma_i)\}. 
		\end{equation}
		Now in the minima we have $v(\tau-\sigma_i) \geq -h(\tau-\sigma_i) \geq -h(E)$ by (\ref{456x2}), and using (\ref{456a8}) we see that each minimum is at least $-h(E)$. Thus (\ref{456a10}) gives
		$$v(\mu-\mu_0)\leq ev(\zeta-\tau)+(n-e)h(E)+(n-1)h(E)\leq 2v(\zeta-\tau)+2nh(E).$$
		The lemma now follows.
	\end{proof}
	
	\subsection{Proof of Theorem \ref{qi}}\label{thmqi}
	
	Take any $O\neq P=(\xi,\eta)$ in $E(\cal K)$ and pick $Q_0=(\mu_0,0)$ in $E(\cal K)$. If $P=Q_0$ then $\log|\xi|_v \leq h(\xi) \leq h(E)/2$, so we get a much stronger bound.
	
	So we may assume $P \neq Q_0$. By Lemma \ref{l2} we get for $Q=P+Q_0=(\mu,\nu) \neq O,Q_0$ the inequality
	\begin{equation}\label{456aa}
	\max\{0,v(\mu-\mu_0)\} \geq \log|\xi|_v-2h(E).
	\end{equation}
	
	Next we pick an odd integer $m \geq 3$ and try to apply Lemma \ref{l3}. Now it may not be true that $E[2m]$ lies in $E(\cal K)$. But certainly it lies in $E(\cal L)$ for some $\cal L$ with $[{\cal L:K}] \leq (2m)^4$. And it may not be true that there is $R$ in $E(\cal K)$ with $mR=Q$. So we fix some $R$ with $mR=Q$ and then $R$ lies in $E(\cal F)$ for some $\cal F$ with $[{\cal F:L}] \leq m^2$. As $Q \neq O$ also $O \neq R=(\zeta,\theta)$.
	
	We take any $w$ of $\cal F$ over $v$ (as always with value group $\bf Z$). By Lemma \ref{l3} with $\cal F$ instead of $\cal K$ we get $O \neq T_w=(\tau_w,\omega_w)$ in $E(\cal F)$ with
	$$\max\{0,w(\zeta-\tau_w)\} \geq {1 \over 2}w(\mu-\mu_0)-m^2h_{\cal F}(E)$$
	for the height $h_{\cal F}$ with respect to $\cal F$.
	Using $h_{\cal F}(E)=nh(E)$ with $n=[{\cal F:K}] \leq 16m^6$ and summing over all $w$ dividing $v$, we deduce for
	$$s=\sum_{w|v}\max\{0,w(\zeta-\tau_w)\}$$
	the lower bound
	$$s \geq {1 \over 2}nv(\mu-\mu_0)-16m^{8}nh(E)$$
	(in which the exponent 8 will eventually play hardly any role). So by (\ref{456aa}) we get
	\begin{equation}\label{456bb}
	s \geq {1 \over 2}n\log|\xi|_v-32m^{8}nh(E).
	\end{equation}
	
	Using Proposition \ref{roth}, with $S$ as the set of $w$ dividing $v$, $A$ as 0 together with the non-zero $\tau_w$, $\e=1/16$ and $\cal F$ in place of $\cal K$, yields two alternatives: either 
	$$h_{\cal F}(\zeta)\le (384 \log 2) l\sum_{w|v}h_{\cal F}(\tau_w)\le 300 l \sum_{w|v}h_{\cal F}(\tau_w)$$
	or
	$$s\le 3h_{\cal F}(\zeta)+ 3(16)^{l}(\chi_{\cal F}+2\sum_{w|v}h_{\cal F}(\tau_w))$$
	for the characteristic $\chi_{\cal F}$ with respect to $\cal F$.
	
	Now there may be up to $n$ different $w$, but as $T_w$ lies in $E[2m]$ we have $l \leq (2m)^4$; further $h_{\cal F}(\tau_w)\le 2h_{\cal F}(E)$.
	
	Next $\hat h_{\cal F}(R)\le 2h_{\cal F}(R)/3+h_{\cal F}(E)/3$ which by (\ref{comp}) is at most $h_{\cal F}(\zeta)+h_{\cal F}(E)/2$. Also $\hat h_{\cal F}(P)=\hat h_{\cal F}(Q)=m^2\hat h_{\cal F}(R)$, whence 
	$$h_{\cal F}(\xi)\le h_{\cal F}(P) \leq {3 \over 2}\hat h_{\cal F}(P)+{1 \over 2}h_{\cal F}(E)\le {3 \over 2}m^2h_{\cal F}(\zeta)+m^2h_{\cal F}(E).$$
	Hence in the first alternative we get
	$$h(\xi)\le 500m^2lnh(E)\le 128000m^{12}h(E).$$
	
	Let us now deal with the second alternative. To estimate $\chi_{\cal F}$ we use Hurwitz in the form
	$$2g_{\cal F}-2=[{\cal F:K}](2g-2)+\sum_w(e_w-1)$$
	(over all $w$ of $\cal F$) for the ramification indices. Here ${\cal F}/{\cal K}$ is unramified for $w$ over the places $v$ of $\cal K$ outside the set $S_E$ of bad reduction. Thus
	$$\sum_w(e_w-1)=\sum_{v \in S_E}\sum_{w|v}(e_w-1) \leq [{\cal F:K}]\sum_{v \in S_E}1.$$
	Also for the discriminant $\Delta=-4a^3c+a^2b^2+18abc-4b^3-27c^2$ we have
	$$\sum_{v \in S_E}1\leq 2h(\Delta) \leq 2h(E)$$
	as in (\ref{456*}). So we get
	$$2g_{\cal F}-2 \leq (2g-2)n+2nh(E).$$
	Also $|S| \leq n$, so $\chi_{\cal F} \leq (2g-1+2h(E))n$. Thus taking into account the $h_{\cal F}(\tau_w)$, we find
	$$s \leq 3h_{\cal F}(\zeta)+3(16)^l(2gn+4n^2h(E)).$$
	Here also using (\ref{comp}) we have
	$$h_{\cal F}(\zeta) \leq {3 \over 2}\hat h_{\cal F}(R)+{1 \over 2}h_{\cal F}(E)={3n \over 2m^2}\hat h(P)+{n \over 2}h(E) \leq {3n \over 2m^2}h(\xi)+nh(E)$$
	so we end up with
	$$s \leq {9n \over 2m^2}h(\xi)+12(16)^ln^2(g+h(E)).$$
	Comparing with (\ref{456bb}), recalling $l \leq 16m^4$ and dividing by $n/2$, we find that
	$$\log|\xi|_v \leq {9 \over m^2}h(\xi)+2^{66m^4}(g+h(E)).$$
		
	Finally to deduce Theorem \ref{qi} take $m$ odd minimal with $9/m^2 \leq \e$. Then indeed $m \geq 3$ so $9/(m-2)^2 > \e$ and
	$$m < 2+3\sqrt{{1 \over \e}}\leq{7/2 \over \sqrt{\e}}.$$
	Thus $66m^4 < 10000/\e^2$ and we are done (even without a little square).

\end{section}

\begin{section}{Comparison of multiplicity bounds}\label{cmb}
	Here we work out some examples of Theorem \ref{Thm:Theorem2.6} and Theorem \ref{qi}. The comparison seems to run best with the Legendre curve. At first we will use the point
	$$P=(2,\sqrt{4-2\lambda})$$
	on $y^2=x(x-1)(x-\lambda)$. For each positive integer $n$ there are coprime polynomials $A_n,B_n$ in ${\bf Z}[\lambda]$ such that $A_n(\lambda)/B_n(\lambda)$ is the abscissa of $n(2,\sqrt{4-2\lambda})$. For example with $n=1,2,3,4$ we get
	$$2,~~-{(\lambda-4)^2 \over 8(\lambda-2)},~~{2(5\lambda^2-16\lambda+16)^2 \over (\lambda^2+8\lambda-16)^2},~~-{(\lambda^4-80\lambda^3+352\lambda^2-512\lambda+256)^2 \over 32(\lambda-2)\lambda^2(3\lambda^2-16\lambda+16)^2}.$$
	With a slightly laborious induction on the standard formulae relating the abscissae of $P_1+P_2,P_1-P_2,P_1,P_2$, one can check that the degree of $A_n$ is $(n^2-1)/2$ (odd $n$) and $n^2/2$ (even $n$) and the degree of $B_n$ is $(n^2-1)/2$ (odd $n$) and $(n^2-2)/2$ (even $n$). Making $n \to \infty$ we see that this confirms the calculation $\hat h(P)=1/2$ of section 3. One can also check that
	\begin{equation}\label{456sq}
	B_n(\lambda)=b_nC_n(\lambda)^2~~({\rm odd}~n),~~B_n(\lambda)=b_n(\lambda-2)C_n(\lambda)^2~~({\rm even}~n)
	\end{equation}
	for $b_n$ in $\bf Z$ and $C_n$ in ${\bf Z}[\lambda]$. 
	
	For every $\lambda_0 \neq 0,1$ we will estimate from above 
	$$w_n(\lambda_0)={\rm ord}_{\lambda=\lambda_0}B_n(\lambda)$$
	assuming it is positive.
	
	First we use Theorem \ref{Thm:Theorem2.6}. The curve $B$ is given by $\mu^2=4-2\lambda$, with typical point $b=(\lambda,\mu)$ (say), and the section $\sigma$ is given by $\sigma(b)=(2,\mu)$. We are considering a particular point $b_0=(\lambda_0,\mu_0)$ for $\mu_0=\sqrt{4-2\lambda_0}$. By the discussion just before Theorem \ref{Thm:Theorem2.6} we have $m_\sigma(b_0)=v(x_{n\sigma}/y_{n\sigma})$ for the valuation $v$ on ${\C}(B)={\C}(\lambda,\mu)$ corresponding to $b_0$. Here coprimality gives
	\begin{equation}\label{456xx}
	v(x_{n\sigma})=v\left({A_n(\lambda) \over B_n(\lambda)}\right)=-w_n(\lambda_0)
	\end{equation}
	at least if $\lambda_0 \neq 2$; while if $\lambda_0=2$ it is $-2w_n(2)$. So $v(x_{n\sigma})<0=\min\{0,v(\lambda)\}$, and it follows easily that 	     $v(y_{n\sigma})=3v(x_{n\sigma})/2$. We conclude
	$$m_\sigma(b_0)={1 \over 2}w_n(\lambda_0)$$
	if $\lambda_0 \neq 2$; and if $\lambda_0=2$ it is $w_n(2)$.
	
	If we fix the tangent space by ${\rm d}z={\rm d}x/y$, then the operator $\Xi$ is given by
	\begin{equation}\label{456PF}
	\Xi(x,y)=4\lambda(1-\lambda)\left(D\left({Dx \over y}\right)+{Dx \over 2(x-\lambda)y}\right)+4(1-2\lambda){Dx \over y}+{2x(x-1) \over (x-\lambda)y}
	\end{equation}
	with $D={\rm d}/{\rm d}\lambda$ (even for $x,y$ in the algebraic closure of ${\C}(\lambda)$; and the last term could be taken as $2y/(x-\lambda)^2$ as well). This corrects (2) of Manin \cite{M} (p.1397).
	
	We find
	\begin{equation}\label{456PF2}
	\Xi(\sigma)={2\mu \over (2-\lambda)^2}.
	\end{equation}
	So $v(\Xi(\sigma))=0$
	if $\lambda_0 \neq 2$ and $v(\Xi(\sigma))=-3$ if $\lambda_0 = 2$. We deduce $m_\sigma(b_0) \leq 2$
	for all $b_0$; and that there are at most finitely many $b_0$ with $m_\sigma(b_0) \geq 2$ (that is, $m_\sigma(b_0) = 2$).
	
	It follows that
	\begin{equation}\label{456yy}
	w_n(\lambda_0) \leq 4
	\end{equation}
	if $\lambda_0 \neq 2$, and $w_n(2) \leq 2$; and that there are at most finitely many $\lambda_0$ with $w_m(\lambda_0) \geq 4$ (that is, $w_n(\lambda_0) = 4$).
	
	In turn this implies that the $C_n$ in (\ref{456sq}) are squarefree apart from at most finitely many squared factors $(\lambda-\lambda_0)^2$. We found no such $\lambda_0$ for $n=1,2,\ldots,20$.
	
	Next we use Theorem \ref{qi} for ${\cal K}={\C}(B)={\C}(\lambda,\mu)$ and $P=(\xi,\eta)=n(2,\mu)$, with $v$ corresponding to $(\lambda_0,\mu_0)$ as above. We find that
	$$v(\xi)=-w_n(\lambda_0)$$
	if $\lambda_0 \neq 2$ as in (\ref{456xx}). Also the logarithmic height of $\xi$ with respect to ${\C}(\lambda)$ is at most $n^2/2$, so with respect to ${\C}(\lambda,\mu)$ we get $h(\xi) \leq n^2$.
	
	As the genus $g=0$ and $h(E)=12$ we obtain
	$$w_n(\lambda_0) \leq \e n^2+12(2^{10000/\e^2})$$
	whenever $0<\e \leq 1/16$. Choosing say
	$$\e=\left({10000\log 2 \over \log n}\right)^{1/2}$$
	for $n \geq 2^{2560000}$, we get
	$$w_n(\lambda_0) \leq 84{n^2 \over (\log n)^{1/2}}+12n \leq 100{n^2 \over (\log n)^{1/2}}.$$
	This is a lot worse than (\ref{456yy}). However we looked at only the points $nP$ on $E({\cal K})$, whereas this group contains ${\bf Z}P+E[2]$. It is also conceivable that the rank exceeds 1, say $E({\cal K})={\bf Z}P+{\bf Z}Q+\cdots$, and then similar arguments would apply to $mQ$ or even $nP+mQ$, giving an explicit $o(n^2+m^2)$. But as mentioned, the determination of $E({\cal K})$ is not yet an effective procedure.
	
	We can allow ourselves an extra $Q$ simply by increasing the field to get
	$$P=(2,\sqrt{4-2\lambda}),~~~Q=(3,\sqrt{18-6\lambda}).$$
	Now $B$ is given by $\mu^2=4-2\lambda,~\nu^2=18-6\lambda$, and we have a second section $\tau$, and we apply Theorem \ref{Thm:Theorem2.6} to $n\sigma+m\tau$. Here $\Xi(\tau)={2\nu / (3-\lambda)^2}$ as in (\ref{456PF2}), and so $\Xi(n\sigma+m\tau)=\beta$ for
	$$\beta={2n\mu \over (2-\lambda)^2}+{2m\nu \over (3-\lambda)^2}.$$
	It is not hard to show that the order at any point $b$ of $B$ is bounded independently of $b,n,m$. In fact  one finds for ``generic'' $n,m$ that $h(\beta)=12$ so $h(\beta)\leq12$ for all $n,m$. Thus $v(\beta) \leq 12$ too, and the corresponding order of the abscissa (now a rational function of $\lambda,\mu,\nu$) is at most 14.
	
	Theorem \ref{qi} gives as above an explicit estimate $o(n^2+m^2)$; but again it can be applied to the full $E({\cal K})={\bf Z}P+{\bf Z}Q+\cdots$, whose generators may now be rather difficult to find.
	
	We finish by proving (\ref{456*a}) and that the 4 is best possible. From
	$$\Xi(x,-y)=\Xi(-(x,y))=-\Xi(x,y)$$
	(or (\ref{456PF}) directly) we see that $\Xi(x,y)$ is an odd function, and we calculate it as $\Upsilon/y^3$, where
	\begin{equation}\label{456de}
	\Upsilon=f_0D^2x+f(Dx)^2+f_1Dx+f_2 
	\end{equation}
	and $f_0,f,f_1,f_2$ are polynomials in $x$ of degrees 3,2,3,4 respectively, with coefficients in ${\bf Z}[\lambda]$. Now (\ref{456*a}) is clear (and the right-hand side of (\ref{456de}) provides a differential equation vanishing at the abscissae of all points of finite order at least 3).
	
	To see that the factor 4 is best possible, we write
	$$\Xi(x,y)^2={\Upsilon^2 \over x^3(x-1)^3(x-\lambda)^3}$$
	and we calculate this for abscissae $x=\xi=\lambda^d+6\lambda+70 ~(d \geq 2)$ and corresponding ordinate $y=\eta$. We find
	\begin{equation}\label{456rat}
	{P_d(\lambda) \over (\lambda^d+6\lambda+70)^3(\lambda^d+6\lambda+69)^3(\lambda^d+5\lambda+70)^3} 
	\end{equation}
	where $P_d$ is in ${\bf Z}[\lambda]$ with leading term $4(d-1)^4\lambda^{8d}$. Thus a lower bound
	$$h(\Xi(\xi,\eta)) \geq 4dh(\lambda)-O(1) \geq 4h(\xi)-O(1)$$
	will follow as soon as the numerator and denominator in (\ref{456rat}) have no common zeroes.
	
	Now we find that the numerator
	$$P_d(\lambda)={P(d,\lambda,\lambda^d) \over \lambda^2}$$
	for a fixed polynomial $P$. If this had a zero in common with say $\lambda^d+6\lambda+70$ in the denominator, then so would $P(d,\lambda,-6\lambda-70)$. This is in ${\bf Z}[\lambda]$ with leading term
	$$5184(d-1)^4\lambda^8.$$
	So the zero $\lambda$ would have to be of degree at most 8 over $\bf Q$. However by Eisenstein with prime 2 we see that it has degree $d$. Therefore when $d \geq 9$ there is no such common zero.
	
	Similar arguments work with the other factors $\lambda^d+6\lambda+69$ and $\lambda^d+5\lambda+70$ in the denominator, for which the Eisenstein primes 3 and 5 suffice. We get leading terms
	$$5184(d-1)^4\lambda^8,~~2500(d-1)^4\lambda^{10}$$
	but in the second case there is a factor $\lambda^2$. So again when $d \geq 9$ there is no common zero.
	
\end{section}

\end{document}